\documentclass[final]{siamltex}

\usepackage{amsmath}
\usepackage{amssymb}
\usepackage{graphicx}
\usepackage{color}

\newcommand{\calU}{\mathcal{U}}

\newcommand{\calE}{\mathcal{E}}
\newcommand{\calF}{\mathcal{F}}
\newcommand{\calB}{\mathcal{B}}

\newcommand{\bbR}{\mathbb{R}}
\newcommand{\bbC}{\mathbb{C}}
\newcommand{\tr}{\operatorname{tr}}

\newcommand{\normg}[1]{\| {#1} \|_{\textrm{glob}}}

\newtheorem{example}{Example}
\newtheorem{remark}{Remark}

\title{Localization theorems for nonlinear eigenvalue problems\thanks{
Supported in part by the Sloan Foundation.
}}

\author{David Bindel\thanks{Department of Computer Science, Cornell
    University, Ithaca, NY 14850 ({\tt bindel@cs.cornell.edu})} \and
    Amanda Hood\thanks{Center for Applied Mathematics, Cornell
    University, Ithaca, NY 14850 ({\tt ah576@cornell.edu})}}

\begin{document}

\maketitle

\begin{abstract}
  Let $T : \Omega \rightarrow \bbC^{n \times n}$ be a matrix-valued
  function that is analytic on some simply-connected domain $\Omega
  \subset \bbC$.  A point $\lambda \in \Omega$ is an eigenvalue if the
  matrix $T(\lambda)$ is singular.  In this paper, we describe new
  localization results for nonlinear eigenvalue problems that
  generalize Gershgorin's theorem, pseudospectral inclusion theorems,
  and the Bauer-Fike theorem.  We use our results to analyze three
  nonlinear eigenvalue problems: an example from delay differential
  equations, a problem due to Hadeler, and a quantum resonance
  computation.
\end{abstract}

\begin{keywords} 
nonlinear eigenvalue problems,
pseudospectra,
Gershgorin's theorem, perturbation theory
\end{keywords}

\begin{AMS}
15A18, % Eigenvalues, singular values, and eigenvectors
15A42, % Inequalities involving eigenvalues and eigenvectors
15A60, % Norms of matrices, numerical range, applications of
       % functional analysis to matrix theory
30E10  % Approximation in the complex domain
\end{AMS}

\pagestyle{myheadings}
\thispagestyle{plain}
\markboth{D. Bindel and A. Hood}{Localization for 
      nonlinear eigenvalue problems}

\section{Introduction}
\label{sec-intro}

In this paper, we study the nonlinear eigenvalue problem
\begin{equation} \label{eq-nep}
  T(\lambda) v = 0, \quad v \neq 0,
\end{equation}
where $T : \Omega \rightarrow \bbC^{n \times n}$ is analytic on a
simply-connected domain $\Omega \subset \bbC$.  Problem~\eqref{eq-nep}
occurs in many
applications~\cite{Betcke:2013:NLEVP,Mehrmann:2005:NEP}, often from
applying transform methods to analyze differential and difference
equations.  The best-studied nonlinear eigenvalue problems are those
for which $T$ is polynomial in
$\lambda$~\cite{Bini:2012:LEMP,Gohberg:2009:MP,Higham:2003:BEMP}, and
particularly those that are quadratic in
$\lambda$~\cite{Sleijpen:1996:QEN,Tisseur:2001:QEP}.  More general
nonlinear eigenvalue problems that involve algebraic or transcendental
matrix functions are prevalent in models with
delay~\cite{Michiels:2007:SST} or
radiation~\cite{Igarashi:1995:NCE,Tausch:2000:FMP,Yang:2005:SLS,Yuan:2006:PBG}.

In this paper, we consider {\em localization results} that define
regions in which any eigenvalues must lie.  Localization regions such
as pseudospectra~\cite{Trefethen:2005:SP} and Gershgorin
disks~\cite{Varga:2004:GAC} are widely used in the analysis of
ordinary eigenvalue problems.  In error analysis, localization results
bound how much numerically computed eigenvalues are affected by
roundoff and other approximations, particularly when the approximation
error is not tiny or the eigenvalues in question are ill-conditioned.
Localization regions that are crude but easy to compute are used in
linear stability of dynamical systems, as an easy way to see that a
matrix has no eigenvalues in the right half plane or outside the unit
disk.  Crude localization results are also used to find good shifts
for spectral transformations commonly used with iterative
eigensolvers.  But though localization is as useful for 
nonlinear eigenvalue problems
as for linear eigenvalue problems, little has been done to
adapt standard localization results to the nonlinear case.

Just as one can localize eigenvalues of an ordinary problem by
localizing zeros of a characteristic polynomial, a standard approach
to localizing eigenvalues of $T(z)$ is to localize the zeros of the
scalar function $\det T(z)$.  Apart from some work in the context of
delay differential equations~\cite{Jarlebring:2008:SDD}, we are not
aware of any efforts to extend localization results that work directly
with the matrix, such as Gershgorin's theorem or the Bauer-Fike
theorem, to the general nonlinear case.  However, related work has
been done for certain instances of~\eqref{eq-nep}.  For polynomial
eigenvalue problems in particular, several researchers have explored
perturbation
theory~\cite{Ahmad:2011:PAC,Bora:2009:SEC,Chu:2003:PEMPBF,Dedieu:2003:PTHPEP,Higham:2007:BEP,Tisseur:2000:BECPEP}
and localization theorems that generalize results for scalar
polynomials (e.g. Pellet's
theorem)~\cite{Bini:2012:LEMP,Melman:2013:GVPMP}, though these results
are of limited use outside the polynomial case.  Similarly, research
into pseudospectra for nonlinear
problems~\cite{Cullum:2001:PAN,Green:2006:PDD,Higham:2002:MPPEPACT,Michiels:2006:PSR,Michiels:2007:SST,Tisseur:2001:SPPEPA,Wagenknecht:2008:SPN}
has primarily focused on specific types of eigenvalue problems, such
as polynomial problems or problems arising from delay differential
equations.

The rest of this paper is organized as follows.  In
Section~\ref{sec-background}, we recall a useful result from the
theory of analytic matrix-valued functions and some background on
subharmonic functions.  In Section~\ref{sec-nl-gg}, we describe a
generalized Gershgorin theorem for nonlinear eigenvalue problems, and
in Section~\ref{sec-pseudospectra}, we introduce and discuss a
nonlinear generalization of pseudospectra.  We then turn to the useful
special case of linear functions with nonlinear perturbations in
Section~\ref{sec-linear}, where we describe analogues of Gershgorin's
theorem and the Bauer-Fike theorem for this case.  We illustrate the
usefulness of our bounds through some examples in
Section~\ref{sec-example}, and conclude in
Section~\ref{sec-conclusion}.

\section{Preliminaries}
\label{sec-background}

We assume throughout this paper that $\Omega \subset \bbC$ is a
simply-connected domain and $T : \Omega \rightarrow \bbC^{n \times n}$
is analytic and {\em regular}, i.e. $\det(T(z)) \not \equiv 0$.  For
$T$ regular, the zeros of $\det(T(z))$ are a discrete set with no
accumulation points in $\Omega$.  We call $\lambda \in \Omega$ a
eigenvalue with multiplicity $m$ if $\det(T(z))$ has a zero of order
$m$ at $\lambda$.  The set of all eigenvalues of the matrix-valued
function $T$ is the spectrum $\Lambda(T)$.  Note that, for simplicity,
we have deliberately restricted our attention to finite eigenvalues.
As with standard eigenvalue problems, when we count eigenvalues in a
region, we always count multiplicity.  If $\Gamma \subset \bbC$ is a
simple closed contour and $T(z)$ is nonsingular for all $z \in
\Gamma$, the number of eigenvalues inside $\Gamma$ is given by the
winding number
\[
  W_{\Gamma}(\det T(z)) =
  \frac{1}{2\pi i} \int_{\Gamma} \left[ \frac{d}{dz} \log \det(T(z))
    \right] \,dz =
  \frac{1}{2\pi i} \int_{\Gamma} \tr\left( T(z)^{-1} T'(z) \right) \,dz.
\]
The following counting argument based on the winding number underpins
most of the results in this paper.

\begin{lemma}
\label{th-A-sE}
  Suppose $T : \Omega \rightarrow \bbC^{n \times n}$ and $E : \Omega
  \rightarrow \bbC^{n \times n}$ are analytic,
  and that $\Gamma \subset \Omega$ is a simple closed contour.  If
  $T(z) + sE(z)$ is nonsingular for all $s \in [0,1]$ and all $z \in
  \Gamma$, then $T$ and $T + E$ have the same number of eigenvalues
  inside $\Gamma$, counting multiplicity.
\end{lemma}
\begin{proof}
  Define $f(z; s) = \det(T(z) + sE(z))$.  The winding number of 
  $f(z; s)$ around $\Gamma$ is the number of 
  eigenvalues of 
  $T+sE$ inside $\Gamma$.  For $z \in \Gamma$ and $s \in [0,1]$, by
  hypothesis, $T(z) + s E(z)$ is nonsingular, and so $f(z; s) \neq 0$.
  Hence, the winding number is continuously defined (and thus
  constant) for $s \in [0,1]$.
\end{proof}

\begin{remark} \rm
Lemma~\ref{th-A-sE} is almost a special case of an operator
generalization of Rouch\'e's theorem due to Gohberg and
Sigal~\cite{Gohberg:1971:OGL}.  However, where Gohberg and Sigal
ensured nonsingularity of $T(z)+sE(z)$ for $z \in \Gamma$ by requiring
$\|T(z)^{-1} E(z)\| < 1$ for some operator norm, in this paper we
consider other tests of nonsingularity.
\end{remark}

In Theorem~\ref{th-nl-gerschgorin} and
Proposition~\ref{th-at-least-one}, we also make use of the theory of
subharmonic functions; see~\cite[Ch.~17]{Rudin:1987:RCA}.  Recall that an upper
semicontinuous function $\phi : \Omega \rightarrow \bbR$ is
subharmonic at $z$ if for any small enough $r > 0$,
\[
  \phi(z) \leq \frac{1}{2\pi} \int_0^{2\pi} \phi(z+re^{i\theta}) \, d\theta.
\]
It immediately follows that subharmonic functions
obey a maximum principle: if $\phi$ is subharmonic on a compact set, 
the maximum occurs on the boundary.
%% %% AH
%% %% Are these the right words?
%% % The set of subharmonic functions is closed under addition, 
%% % maximization, 
%% The set of subharmonic functions is closed under the binary operations of
%% addition and pointwise maximum 
%% and 
%% %% AH
%% %% uniform limits?
%% % uniformly convergent limits, 
%% under uniform limits,
%% and
%% if $f$ is 
%% %% AH
%% %a holomorphic function 
%% holomorphic
%% at $z$, then $|f|$ and $\log |f|$ are
%% subharmonic at $z$.
%
% DSB: 
% I like the original text better than the revision.  But if we're
% trying to minimize ambiguity, how about this?  In the revised text
% below think it's fine to just write ``max'' and not specify that
% it's pointwise -- it wouldn't really parse right otherwise.
%
If $f$ is holomorphic at $z$, then $|f|$ and $\log |f|$ are
subharmonic at $z$; 
if $\phi$ and $\psi$ are subharmonic, then so are
$\phi+\psi$ and $\max(\phi, \psi)$; 
and if $\phi_j$ is a sequence of
subharmonic functions that converges uniformly to a limit $\phi$, then
$\phi$ is also subharmonic.  
We can write any vector norm as $\|v(z)\| = \max_{l^* \in \calB^*}
|l^* v(z)|$ where $\calB^*$ is an appropriate unit ball in the dual
space; hence, if $v$ is a vector-valued holomorphic function, then
$\|v\|$ and $\log \|v\| = \max_{l^* \in \calB^*} \log |l^* v|$ are
also subharmonic.

\section{Gershgorin bounds for nonlinear problems}
\label{sec-nl-gg}

Lemma~\ref{th-A-sE} provides a template for constructing inclusion
regions to compare the spectra of two related problems.  The following
is a nonlinear generalization of Gershgorin's theorem that allows us
to compare the spectrum of a general matrix-valued function to the
zeros of a list of scalar-valued functions.

\begin{theorem} \label{th-nl-gerschgorin}
  Suppose $T(z) = D(z) + E(z)$ where $D, E : \Omega \rightarrow
  \bbC^{n \times n}$ are analytic and $D$ is diagonal.  Then for any
  $0 \leq \alpha \leq 1$,
  \[
    \Lambda(T) \subset \bigcup_{j=1}^n G_j^{\alpha},
  \]
  where $G_j^{\alpha}$ is the $j$th generalized Gershgorin region
  \[
    G_j^{\alpha} = 
    \{ z \in \Omega : |d_{jj}(z)| \leq r_j(z)^{\alpha} c_j(z)^{1-\alpha} \}
  \]
  and $r_j$ and $c_j$ are the $j$th absolute row and column sums of $E$, i.e.
  \[
    r_j(z) = \sum_{k=1}^n |e_{jk}(z)|, \quad
    c_j(z) = \sum_{i=1}^n |e_{ij}(z)|.
  \]
  Moreover, suppose that $\calU$ is a bounded connected component of 
  the union $\bigcup_j G_j^{\alpha}$ such that 
  $\bar{\calU} \subset \Omega$.
  Then $\calU$ contains the same number of eigenvalues of $T$ and $D$;
  and if $\calU$ includes $m$ connected components of the Gershgorin
  regions, it must contain at least $m$ eigenvalues.
\end{theorem}
\begin{proof}
  If $z \in \Omega$ is not in $\bigcup_j G_j^{\alpha}$, then, 
  for each $j$,
  \[
    |d_{jj}| 
    > r_j^{\alpha} c_j^{1-\alpha} \\
    = (\hat{r}_j+|e_{jj}|)^\alpha \, (\hat{c}_j+|e_{jj}|)^{1-\alpha},
  \]
  where 
  \[
    \hat{r}_j = \sum_{k \neq j} |e_{jk}|, \quad
    \hat{c}_j = \sum_{i\neq j} |e_{ij}|
  \]
  are the deleted absolute row and column sums of $E$.  Applying
  H\"older's inequality with $p = 1/\alpha$ and $q = 1/(1-\alpha)$, we
  have
  \[
    |d_{jj}| 
    ~~>~~
    (\hat{r}_j^{\alpha p}+|e_{jj}|^{\alpha p})^{1/p} \,
    (\hat{c}_j^{(1-\alpha)q}+|e_{jj}|^{(1-\alpha)q})^{1/q}
    ~~\geq~~ 
    \hat{r}_j^\alpha \hat{c}_j^{1-\alpha} + |e_{jj}|,
  \]
  and by the triangle inequality,
  \[
    |d_{jj} + e_{jj}| 
    \geq |d_{jj}|-|e_{jj}| 
    > \hat{r}_j^\alpha \hat{c}_j^{1-\alpha}.
  \]
  Therefore, for each $j$,
  \[
    |t_{jj}| > \left( \sum_{k \neq j} |t_{jk}| \right)^\alpha
               \left( \sum_{i \neq j} |t_{ij}| \right)^{1-\alpha},
  \]
  and so by a nonsingularity test of 
  Ostrowski~\cite[Theorem 1.16]{Varga:2004:GAC}, 
  $T(z)$ is nonsingular.  The same
  argument shows that $D(z) + sE(z)$ is nonsingular for any 
  $0 \leq s \leq 1$.

  Because $D + sE$ is nonsingular outside the Gershgorin regions,
  Lemma~\ref{th-A-sE} implies that any closed contour in $\Omega$
  that does not pass through $\bigcup_j G_j^{\alpha}$ contains the same
  number of eigenvalues from $\Lambda(T)$ and $\Lambda(D)$, counting
  multiplicity.  Thus, if $\calU$ is a bounded connected component of 
  $\bigcup_j G_j^{\alpha}$, $\bar{\calU} \subset \Omega$, then $D$
  and $T$ must have the same number of eigenvalues inside $\calU$.

  To establish the final counting result, we now show that $d_{jj}$
  has at least one zero in each bounded connected component of
  $G_j^\alpha$ whose closure is in $\Omega$.  Define vector-valued
  functions $v$ and $w$ by $v_k = e_{jk} / d_{jj}$ and $w_k = e_{kj} /
  d_{jj}$, and note that
  \[
  G_{j}^{\alpha} 
   = \{ z \in \Omega : \phi(z) \geq 0 \}, \quad
  \phi(z) \equiv \alpha \log \|v(z)\|_1 + (1-\alpha) \log \|w(z)\|_1.
  \]
  Let $K \subset \Omega$ be the closure of a connected component of
  $G_j^{\alpha}$, and define
  \[
    K_{\epsilon} = 
    \bigcup_{z \in K} \{ z' \in \bbC : |z-z'| \leq \epsilon \}.
  \]
  For small enough $\epsilon$, we know that $K_{\epsilon}$ lies
  within $\Omega$ 
  %% AH
  %% Why?
  % DSB: Because K lies in \Omega (by hypothesis), K is closed,
  %   and Omega is open.  Thus, there's a little ball around each
  %   point in K that lives entirely in Omega.
  and does not intersect any other connected
  components, so the maximum value of $\phi(z)$ on $K_{\epsilon}$
  does not occur on the boundary.  Therefore, $\phi(z)$ cannot
  be subharmonic on $K_{\epsilon}$; but it would be subharmonic
  on $K_{\epsilon}$ if $d_{jj}$ had no zeros inside $K_{\epsilon}$.
  Thus, there must be at least one zero of $d_{jj}$ inside $K_{\epsilon}$,
  and hence in $K = \bigcap_{\epsilon} K_{\epsilon}$.
\end{proof}

The usual statement of Gershgorin's theorem corresponds to the special
case when $T(z) = A-zI = (D-zI) + E$, where $D$ is the diagonal part
of $A$ and $\alpha$ is set to zero or one.  Then the Gershgorin
regions are simply disks, and a component consisting of $m$ disks
contains $m$ eigenvalues.  However, Theorem~\ref{th-nl-gerschgorin}
involves some hypotheses that are not needed for the standard version
of Gershgorin's theorem.  We illustrate the role of these hypotheses
through three examples.

\begin{example} \rm
For the matrix
\[
  T(z) = \begin{bmatrix} 1 & z \\ 0 & z \end{bmatrix},
\]
we have Gershgorin regions $G_{1}^1 = \{0\}$ and 
$G_2^1 = \{ z : |z| \geq 1 \}$.  
The first region contains the sole eigenvalue for the problem.  The
fact that the second region contains no eigenvalues does not violate
the counting result in Theorem~\ref{th-nl-gerschgorin}, since the
second region is unbounded.
\end{example}

\begin{example} \rm
\label{ex2}
Consider the matrix
\[
  T(z) = 
  \begin{bmatrix} 
    z & 1 & 0 \\ 
    0 & z^2-1 & 0.5 \\ 
    0 & 0 & 1 
  \end{bmatrix}.
\]
The Gershgorin regions are shown in Figure~\ref{fig-ex23} (left).
For this problem, $G_1^1$ is the closed unit disk, $G_2^1$ consists of
two roughly circular components around $\pm 1$, and $G_3^1$ is empty.
The region $\calU = G_1^1 \cup G_2^1$ intersects two Gershgorin
regions, and contains $3 > 2$ eigenvalues.  
Unlike Gershgorin disks in the standard problem, each
bounded Gershgorin region may contain one or many eigenvalues.
\end{example}

\begin{figure}
\begin{center}
\includegraphics[width=0.45\textwidth]{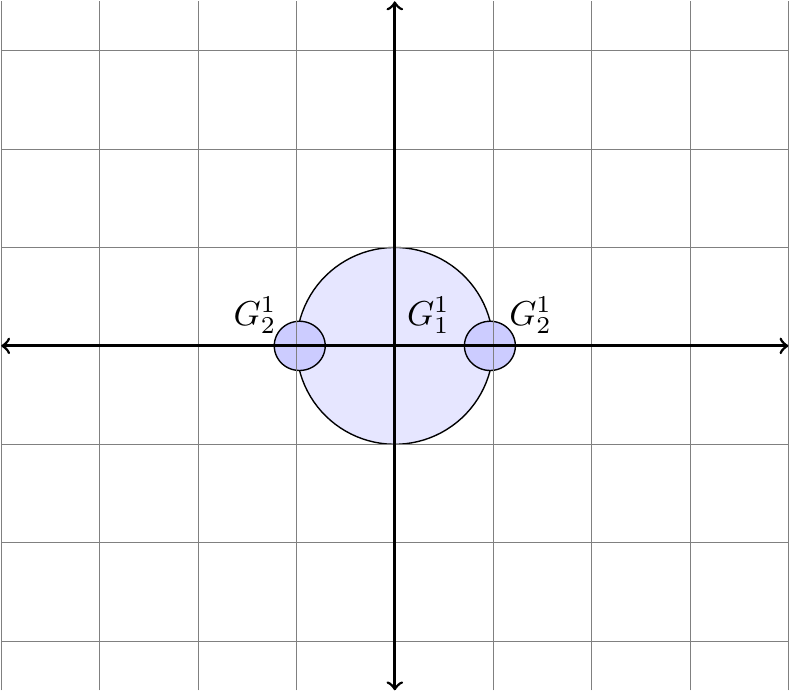}
\includegraphics[width=0.45\textwidth]{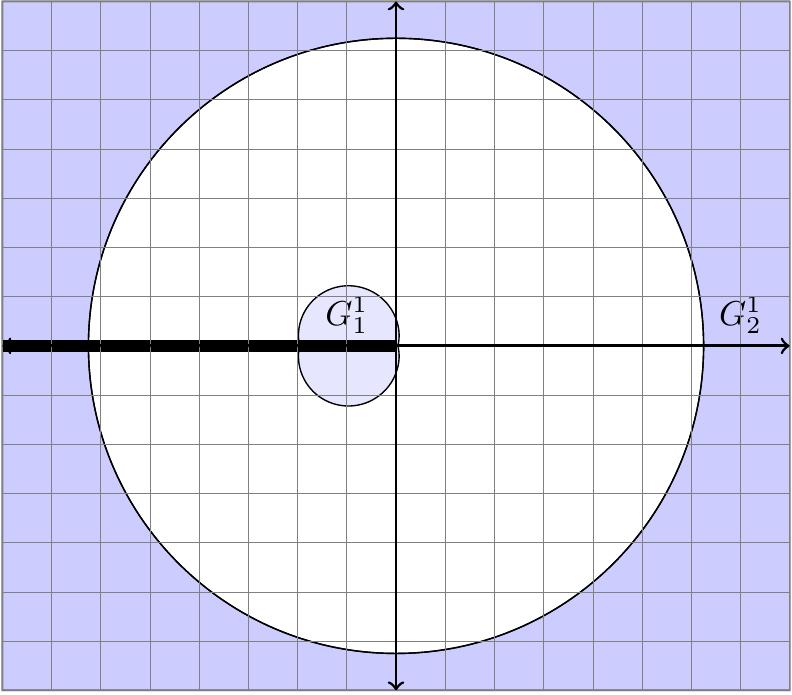}
\end{center}
\caption{Gershgorin regions in Example~\ref{ex2} (left) and
  Example~\ref{ex3} (right).  In Example 2, the second Gershgorin
  region consists of two pieces, and the union of $G_1^1$ and $G_1^2$
  contains three eigenvalues.  In Example 3, neither Gershgorin region
  contains eigenvalues; but both $\bar{G_1^1}$ and $\bar{G_2^2}$
  intersect $(-\infty,0]$, which is not in the domain $\Omega$ for
  this problem.}
\label{fig-ex23}
\end{figure}

\begin{example} \rm
\label{ex3}
Consider the matrix
\[
  T(z) = 
  \begin{bmatrix} 
    z - 0.2 \sqrt{z} + 1 & -1 \\
    0.4 \sqrt{z} & 1
  \end{bmatrix}
\]
defined on $\Omega = \bbC - (-\infty,0]$, where $\sqrt{z}$ is taken to
be the principal branch of the square root function.  If we let
$D(z)$ be the diagonal of $T(z)$, the Gershgorin regions are as
shown in Figure~\ref{fig-ex23} (right).  Note that
\[
  \det(D(z)) = z-0.2\sqrt{z} + 1 = 
  (\sqrt{z} - 0.1 - i\sqrt{0.99})(\sqrt{z} - 0.1 + i\sqrt{0.99})
\]
has two solutions on the primary sheet of the square root function,
but
\[
  \det(T(z)) = z+0.2\sqrt{z} + 1 =
  (\sqrt{z} + 0.1 - i\sqrt{0.99})(\sqrt{z} + 0.1 + i\sqrt{0.99})
\]
only has solutions on the second sheet of definition.  Thus, the set
$G_1^1$ contains two eigenvalues of $D(z)$, but no eigenvalues of
$T(z)$.  This does not violate Theorem~\ref{th-nl-gerschgorin},
because the closed set $\bar{G_1^1}$ includes $[-1,0] \not \subset
\Omega$.
\end{example}

\section{Pseudospectral regions} 
\label{sec-pseudospectra}

The spectrum of a matrix $A$ is the complement of the resolvent set,
i.e., the set of $z$ such that the resolvent operator $R(z) =
(zI-A)^{-1}$ is well-defined.  The {\em $\epsilon$-pseudospectrum} of
$A$ is equivalently defined as
\begin{align}
  \Lambda_{\epsilon}(A) 
  & \equiv \{ z : \|R(z)\| > \epsilon^{-1} \}
           \label{linear-ps-def1} \\
  & \equiv \bigcup_{\|E\| < \epsilon} \Lambda(A + E),
           \label{linear-ps-def2}
\end{align}
with the convention that $\|R(\lambda)\| = \infty$ when $\lambda \in
\Lambda(A)$.  

Several authors have worked on nonlinear generalizations of
pseudospectra
~\cite{Cullum:2001:PAN,Green:2006:PDD,Higham:2002:MPPEPACT,Michiels:2006:PSR,Michiels:2007:SST,Tisseur:2001:SPPEPA,Wagenknecht:2008:SPN}.
The usual definitions of 
pseudospectra for nonlinear problems
generalize~\eqref{linear-ps-def2}.  Let $\calF$ be a space consisting
of some set of analytic matrix-valued functions of interest; then
the $\epsilon$-pseudospectrum for $T \in \calF$ is
\begin{equation} \label{nl-ps-def-general}
  \Lambda_{\epsilon}(T) = \bigcup_{ E\in \calF, \normg{E} < \epsilon } \Lambda(T+E).
\end{equation}
where $\normg{E}$ is a global measure of the size of the perturbing
function $E$.
For polynomial eigenvalue problems and nonlinear eigenvalue problems
from the analysis of delay differential equations, many authors use
the definition~\eqref{nl-ps-def-general} with
\begin{equation} \label{nl-ps-def-michiels}
  \calF \equiv \left\{ \sum_{i=0}^m A_i p_i(\lambda) : 
                  A_i \in \bbC^{n \times n} \right\}, \quad
  \normg{\cdot} = \mbox{function of } A_0, A_1, \ldots, A_m
\end{equation}
where the functions $p_i(\lambda)$ are fixed entire 
functions~\cite[Chapter 2]{Michiels:2007:SST}.  However, we wish
to use our results to compare nonlinear eigenvalue problems
with different types of dependencies on $\lambda$; 
for example, we want to compare problems with transcendental
dependence on $\lambda$ to approximations that have polynomial or
rational dependence on $\lambda$.
For this purpose, there may not be a natural formulation in terms of a
standard set of coefficient functions.

We take $\calF$ to be the space of all analytic matrix-valued
functions $C^{\omega}(\Omega, \bbC^{n \times n})$, and measure
size with
\begin{equation} \label{our-norm-g-def}
  \normg{E} \equiv \sup_{z \in \Omega} \|E(z)\|.
\end{equation}
Using the general definition~\eqref{nl-ps-def-general} with
the size measure~\eqref{our-norm-g-def}, we have three
equivalent expressions for the 
pseudospectra,
similar to the equivalent definitions for 
ordinary pseudospectra; see~\cite[Theorem 2.1]{Trefethen:2005:SP}.  
\begin{proposition}
  Let
  $\calE = \{ E : \Omega \rightarrow \bbC^{n \times n} \mbox{ s.t. }
  E \mbox{ analytic},\  \sup_{z\in\Omega} \|E(z)\| < \epsilon \}$
  and $\calE_0 = \{ E_0 \in \bbC^{n \times n} : \|E_0\| < \epsilon \}$.
  Then the following definitions are equivalent:
  \begin{align}
  \Lambda_{\epsilon}(T) & =
    \{z \in \Omega : \|T(z)^{-1}\| > \epsilon^{-1} \}
    \label{eq-nlps-1} \\
  & =
    \bigcup_{E \in \calE} \Lambda(T + E)
    \label{eq-nlps-2} \\
  & = 
    \bigcup_{E_0 \in \calE_0} \Lambda(T+E_0).
    \label{eq-nlps-3}
  \end{align}
\end{proposition}
\begin{proof}
  Denote the sets in~\eqref{eq-nlps-1}, \eqref{eq-nlps-2}, and
  \eqref{eq-nlps-3} as 
  $\Lambda_{\epsilon}^1(T)$,
  $\Lambda_{\epsilon}^2(T)$, and
  $\Lambda_{\epsilon}^3(T)$.  We break the proof into three steps:

$z \in \Lambda_{\epsilon}^2(T) \iff z \in \Lambda_{\epsilon}^3(T)$:~~~ 
If $T(z) + E(z)$ is singular for some $E \in \calE$, then $T(z) + E_0$ is 
singular for $E_0 = E(z)$. Since $E_0 \in \calE_0$, it follows that
$z \in \Lambda_{\epsilon}^3(T)$. Conversely, if $T(z) + E_0$ is singular
for some $E_0 \in \calE_0$, then $T(z) + E(z)$ is singular for $E$ the
constant function $E_0$.

  $z \not \in \Lambda_{\epsilon}^1(T) \implies 
   z \not \in \Lambda_{\epsilon}^3(T)$:~~~
  Suppose $\|T(z)^{-1}\| \leq \epsilon^{-1}$.  Then for any $E_0$ such
  that $\|E_0\| < \epsilon$, we have that $\|T(z)^{-1} E_0\|<1$, 
  so there is a convergent Neumann series for $I + T(z)^{-1} E_0$.
  Thus, $(T(z) + E_0)^{-1} = (I+T(z)^{-1} E_0)^{-1}T(z)^{-1}$ is well defined.
  
  $z \in \Lambda_{\epsilon}^1(T) \implies 
   z \in \Lambda_{\epsilon}^3(T)$:~~~
  Eigenvalues of $T$ belong to both sets, so we need only consider $z
  \in \Lambda^1_{\epsilon}(T)$ not an eigenvalue.  So suppose $T(z)$
  is invertible and $s^{-1} = \|T(z)^{-1}\| > \epsilon^{-1}$.  
  Then $T(z)^{-1} u = s^{-1} v$ for some vectors $u$ and $v$ with
  unit norm; alternately, write $su = T(z) v$.  Let $E_0 = -suw^*$,
  where $w^*$ is a dual vector of $v$.  Then
  $\|E_0\| = s < \epsilon$, and $T(z)+E$ is singular with $v$ as a
  null vector.
\end{proof}

The $\epsilon$-pseudospectra clearly contains the ordinary spectrum,
but we can say more.  The following result is nearly identical to the
analogous statement for ordinary pseudospectra~\cite[Theorem
  4.2]{Trefethen:2005:SP}:
\begin{proposition}
\label{th-at-least-one}
  Suppose $T : \Omega \rightarrow \bbC^{n \times n}$ is analytic and
  $\calU$ is a bounded connected component of $\Lambda_{\epsilon}(T)$ 
  %such that 
  with
  $\bar{\calU} \subset \Omega$.  Then $\calU$
  contains an eigenvalue of $T$.
\end{proposition}
\begin{proof} 
  If $T(z)^{-1}$ is analytic on $\bar{\calU}$, then 
  %% for any $z_0 \in \calU$ and sufficiently small
  %% $\epsilon > 0$,
  %% \begin{align*}
  %%  \|T(z_0)^{-1}\| \le \frac{1}{2\pi} \oint_{|z-z_0| = \varepsilon}
  %%  \frac{\|T(z)^{-1}\|}{|z-z_0|}\,|dz| \le \max_{|z-z_0| = \varepsilon}
  %%  \|T(z)^{-1}\|.
  %% \end{align*}
  $\|T(z)^{-1}\|$ is subharmonic on $\bar{\calU}$.
  Therefore, the maximum of $\|T(z)^{-1}\|$ % over $\bar{\calU}$ must
  must be attained on the boundary.
  But
  $\|T(z)^{-1}\| = \epsilon^{-1}$ for $z \in \partial \calU$, and
  $\|T(z)^{-1}\| > \epsilon^{-1}$ for $z \in \calU$.  Therefore,
  $T(z)^{-1}$ cannot be analytic on $\calU$, i.e. there is an
  eigenvalue in $\calU$.
\end{proof}

A useful feature of pseudospectra is the connection with backward
error, and this carries over to the nonlinear case:
\begin{proposition}
  Suppose $T(\hat{\lambda}) x = r$ and $\|r\|/\|x\| < \epsilon$.
  Then $\hat{\lambda} \in \Lambda_{\epsilon}(T)$.
\end{proposition}
\begin{proof}
  Define $E = -\frac{rx^T}{\|x\|^2}$.  Then 
  $(T(\hat{\lambda}) + E) x = 0$ and $\|E\| = \|r\|/\|x\| < \epsilon$.
\end{proof}

We can also compare eigenvalue problems via pseudospectra.  As
discussed in the next section, this is particularly useful in the case
when one of the problems is linear.
\begin{theorem} \label{th-nep-pseudospec-inclusion}
  Suppose $T : \Omega \rightarrow \bbC^{n \times n}$ and 
  $E : \Omega \rightarrow \bbC^{n \times n}$ are analytic, and let
  \[
    \Omega_{\epsilon} \equiv \{ z \in \Omega : \|E(z)\| < \epsilon \}.
  \]
  Then
  \[
    \Lambda(T+E) \cap \Omega_{\epsilon} ~~\subset~~
    \Lambda_{\epsilon}(T) \cap \Omega_{\epsilon}.
  \]
  Furthermore, if $\calU$ is a bounded connected component of
  $\Lambda_{\epsilon}(T)$ such that $\bar{\calU} \subset
  \Omega_{\epsilon}$, then $\calU$ contains exactly the same number of
  eigenvalues of $T$ and $T+E$.
\end{theorem}
\begin{proof}
  The inclusion result is obvious based on the characterization of the
  pseudospectra as unions of spectra of perturbations to
  $T$.  The counting result follows from the continuity of 
  eigenvalues: the set $\Lambda_{\epsilon}(T) \cap \Omega_{\epsilon}$
  contains $\Lambda(T+sE) \cap \Omega_{\epsilon}$ for all $0 \leq s
  \leq 1$, so for each eigenvalue of $T+E$ in $\calU$, there is a
  continuously-defined path to a corresponding eigenvalue $T$ that
  remains in $\calU$.
\end{proof}

\section{Nonlinear perturbations of linear eigenvalue problems}
\label{sec-linear}
A {\em linearization} of a matrix polynomial $P : \bbC \rightarrow
\bbC^{n \times n}$ is a pair $(A,B) \in \bbC^{(nd) \times (nd)}$ such
that the polynomial $P$ and the pencil $(A,B)$ have the same spectrum
and the same Jordan structure.
There are many possible linearizations, and significant effort has
gone into characterizing linearizations and their structural
properties~\cite{Mackey:2006:SPE,Mackey:2006:VSL}.  More recent work
addresses similar linearizations for rational eigenvalue
problems~\cite{Su:2011:SRE}.  One way to find the spectrum of a
nonlinear matrix function $T$ is to approximate $T$ by some rational or
polynomial function $\hat{T}$, then find eigenvalues of $\hat{T}$
through a linearization.  In this case, the spectrum of $T$ can be
analyzed as a nonlinear perturbation of a linearization of $\hat{T}$.

We follow a simple strategy to generalize standard perturbation
theorems for linear eigenvalue problems to the case where the
perturbations are nonlinear.  Let
$T : \Omega \rightarrow \bbC^{n  \times n}$ have the form
\[
  T(z) = A-zB + E(z),
\]
and suppose we can bound $E$, either in norm or in the magnitude of
individual components, over a domain $\Omega_{E} \subset \Omega$.
We then apply perturbation theorems from the linear case that are
valid for any {\em fixed} perturbation which is similarly controlled.
This argument gives us a set that includes all eigenvalues of $T$
inside $\Omega_E$.  By continuity of the eigenvalues, if $\calU$ is a
bounded connected component such that $\bar{\calU} \subset \Omega_E$,
then $\calU$ contains the same number of eigenvalues of $T$ as of the
linear pencil $A-zB$.

Perhaps the simplest bound of this sort involves the pseudospectra of the
generalized eigenvalue problem:
\begin{corollary}
\label{th-pseudospec-lep}
  Suppose $E : \Omega \rightarrow \bbC^{n \times n}$ is analytic, and
  let
  \[
    \Omega_{\epsilon} \equiv \{ z \in \Omega : \|E(z)\| < \epsilon \}.
  \]
  Suppose also that $(A,B)$ is a regular pencil.  Then for $T = A-zB + E(z)$,
  \[
    \Lambda(T) \cap \Omega_{\epsilon} \subset \Lambda_{\epsilon}(A,B),
  \]
  where $\Lambda_{\epsilon}(A,B)$ denotes the
  $\epsilon$-pseudospectrum for the pencil $A-zB$, i.e.
  \[
    \Lambda_{\epsilon}(A,B) \equiv
    \{ z \in \bbC : \|(A-zB)^{-1}\| > \epsilon^{-1} \}.
  \]
  Furthermore, if $\calU$ is a bounded connected component of
  $\Lambda_{\epsilon}(A,B)$ such that $\bar{\calU} \subset
  \Omega_{\epsilon}$, then $\calU$ contains exactly the same number of
  eigenvalues of $T$ and of the pencil $(A,B)$.
\end{corollary}
\begin{proof}
  This is a special case of Theorem~\ref{th-nep-pseudospec-inclusion}.
\end{proof}

The pseudospectral bound is simple, but computing the pseudospectra
of a pencil may be expensive.  Consequently, we may be better served
by Gershgorin bounds.
\begin{corollary}
\label{th-gerschgorin}
  Suppose 
  \[
    T(z) = D-zI+E(z)
  \]
  where $D \in \bbC^{n \times n}$ is diagonal and $E : \Omega
  \rightarrow \bbC^{n \times n}$ is analytic.  Suppose also that the
  absolute row and column sums of $E$ are uniformly bounded, i.e.
  $\forall z \in \Omega$,
  \[
    \sum_{j=1}^n |e_{ij}(z)| \leq r_i, \quad
    \sum_{i=1}^n |e_{ij}(z)| \leq c_j.
  \]
  Then for any $0 \leq \alpha \leq 1$, the eigenvalues of $T$ lie in 
  $\bigcup_{i=1}^n G_i$, where the $G_i$ are generalized Gershgorin disks
  \[
     G_i \equiv \{ z \in \bbC : |z-d_{ii}| \leq \rho_i \}, \quad
     \rho_i \equiv r_i^{\alpha} c_i^{1-\alpha}.
  \]
  Furthermore, if $\calU$ is a union of $k$ disks which are disjoint
  from the remaining disks, and if $\calU \subset \Omega$, then
  $\calU$ contains exactly $k$ eigenvalues.
\end{corollary}
\begin{proof}
  This is a direct corollary of Theorem~\ref{th-nl-gerschgorin}, 
  noting that in this case $D(z) = D - zI$ has exactly $k$
  eigenvalues in the region $\mathcal{U}$.
\end{proof}

Like the ordinary Gershgorin theorem, 
Theorem~\ref{th-nl-gerschgorin} and Corollary~\ref{th-gerschgorin} 
are particularly powerful in
combination with an appropriate change of basis.  As an
example, we have the following nonlinear version of a well-known
corollary of a theorem due to Bauer and Fike~\cite[Theorem
  IV]{Bauer:1960:NET}:
\begin{theorem} \label{th-bauer-fike}
  Suppose
  \[
  T(z) = A - z I + E(z),
  \]
  where $A \in \bbC^{n \times n}$ has a complete basis of eigenvectors
  $V \in \bbC^{n \times n}$ and $E : \Omega \rightarrow
  \bbC^{n \times n}$ is analytic.  Suppose also that 
  $|E(z)| \leq F$ componentwise for all $z \in \Omega$.  Then the eigenvalues
  of $T$ in $\Omega$ lie in the union of disks
  \[
    \bigcup_{i=1}^n \left\{ z \in \bbC : |z-\lambda_{i}| \leq \phi_i \right\}, \quad
    \phi_i \equiv n\|F\|_2 \sec(\theta_i),
  \]
  where $(\lambda_i, w_i, v_i)$ are eigentriples of $A$ and $\theta_i$
  is the angle between the left and right eigenvectors $w_i$ and
  $v_i$.  If $\calU$ is a union of any $k$ of these disks
  that are disjoint from the remaining disks, and if $\calU \subset \Omega$,
  then $\calU$ contains exactly $k$ eigenvalues of $T$.
\end{theorem}
\begin{proof}
  The proof follows by applying Corollary~\ref{th-gerschgorin} to
  $V^{-1} T(z) V$ and bounding the row sums of $|V^{-1} E(z) V|$.
  Without loss of generality, assume the columns of $V$ are normalized
  to unit Euclidean length.  For any $z \in \Omega$, note that the
  absolute row sum $r_i(z)$ of $V^{-1} E V$ is bounded by
  \begin{align*}
    r_i(z) 
    &=     \sum_{j=1}^n |e_i^* (V^{-1} E(z) V) e_j| \\
    &\leq  e_i^T \, |V^{-1}| \; F \; |V| \, e,
  \end{align*}
  where $e_i \in \bbR^n$ is the $i$th standard basis vector and $e \in
  \bbR^n$ is the vector of all ones.  Let $w_i^* = e_i^*V^{-1}$ be the
  $i$th left eigenvector, and note that the normalization of $V$
  implies that $\|\, |V|e \, \|_2 \leq n$.  Therefore,
  \[
    |r_i(z)| \leq \|e_i^T |V^{-1}|\|_2 \, \|F\|_2 \, \||V| e\|_2
             \leq \|w_i\|_2 \|F\|_2 n.
  \]
  Note that $w_i^* v_i = 1$ and $\|v_i\| = 1$ by the normalization
  conditions, so
  \[
    \|w_i\|_2 = \frac{\|w_i\|_2 \|v_i\|_2}{|w_i^* v_i|} = \sec(\theta_i).
  \]
  Therefore, we have the uniform bound
  \[
    |r_i(z)| \leq \phi_i = n \|F\|_2 \sec(\theta_i).
  \]
\end{proof}

\section{Applications}
\label{sec-example}

In general, the spectrum of a nonlinear eigenvalue problem 
can be more
complicated than that of linear or polynomial eigenvalue problems,
with infinitely many eigenvalues scattered across the complex plane.
The analysis required to localize its spectrum is thus inherently more
involved.  In this section, we give three examples with infinitely
many eigenvalues.  In each case, we use our localization results to
compare eigenvalues of the original problem to those of simpler
problems.  Because different approximating problems yield accurate
eigenvalue estimates in different regions, several approximations may
be necessary to get a complete picture.

\subsection{Hadeler}

\begin{figure}
\begin{center}
\includegraphics[width=0.8\textwidth]{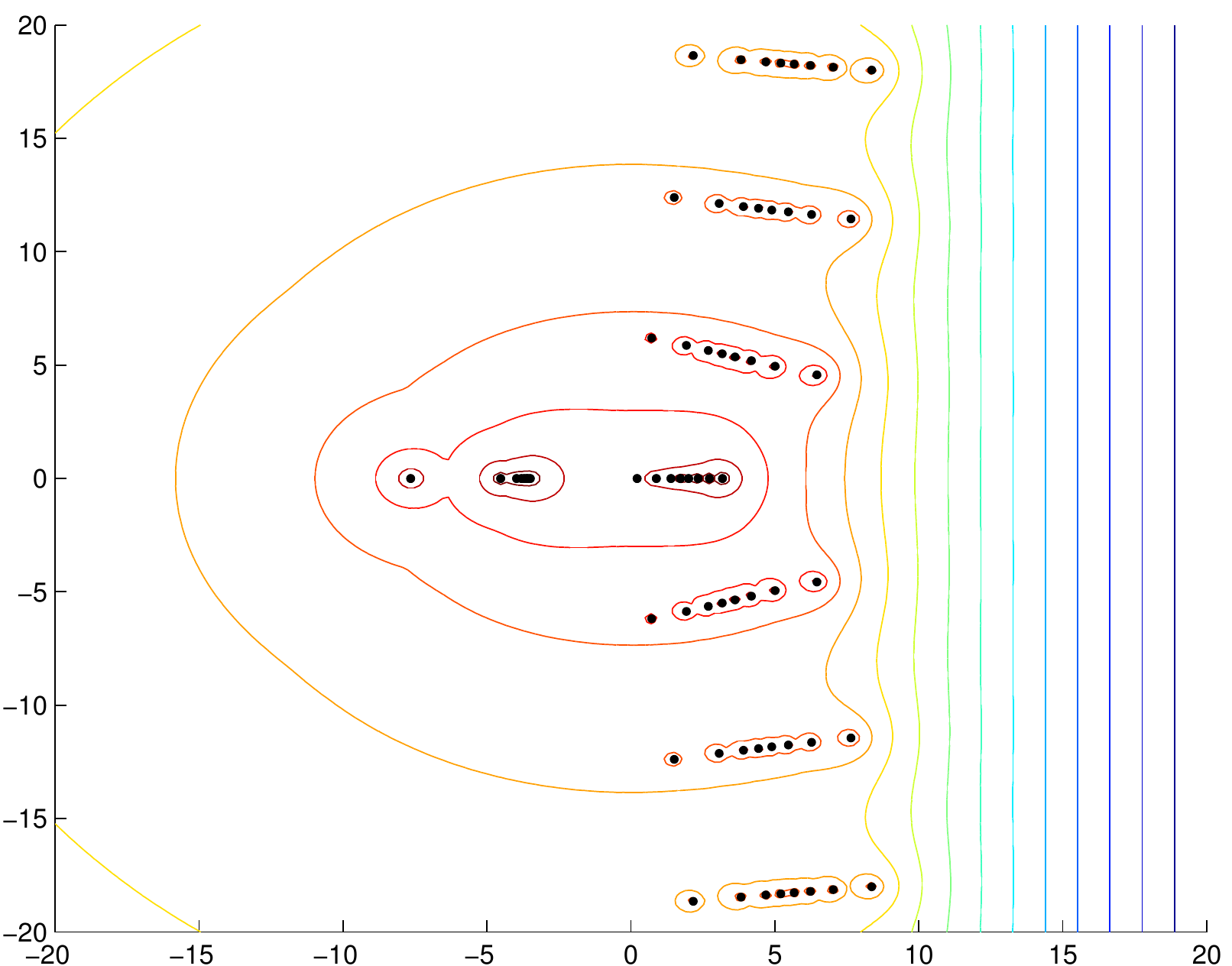}
\end{center}
\caption{Spectrum and pseudospectra for the Hadeler problem.  To
  compute the eigenvalues, we approximate eigenvalues of $T(z)$ by
  eigenvalues of a polynomial interpolating $T(z)$ through Chebyshev
  points on parts of certain curves $z_m(\theta)$ and along the real
  axis, then refine these estimates by a few steps of Newton
  iteration.  }
\label{fig:hadeler1}
\end{figure}

The Hadeler problem in the NLEVP collection has the form
\[
  T(z) = B (\exp(z)-1) + A z^2 - \alpha I
\]
where $A, B \in \bbR^{8 \times 8}$ are real and symmetric positive
definite.  The eigenvalues over part of the complex plane are shown in
Figure~\ref{fig:hadeler1}.  The spectrum consists of sixteen simple
real eigenvalues and infinitely many complex eigenvalues arranged in
groups of eight near certain curves $z_m(\theta)$ described later.  
We use Theorem~\ref{th-nl-gerschgorin} to compare $T(z)$ to two
simpler problems in order to localize both the eigenvalues close to
the real line and those that are farther away.

\subsubsection{Comparison to a polynomial problem}
\label{sec:hadeler2}

\begin{figure}
\begin{center}
\includegraphics[width=0.8\textwidth]{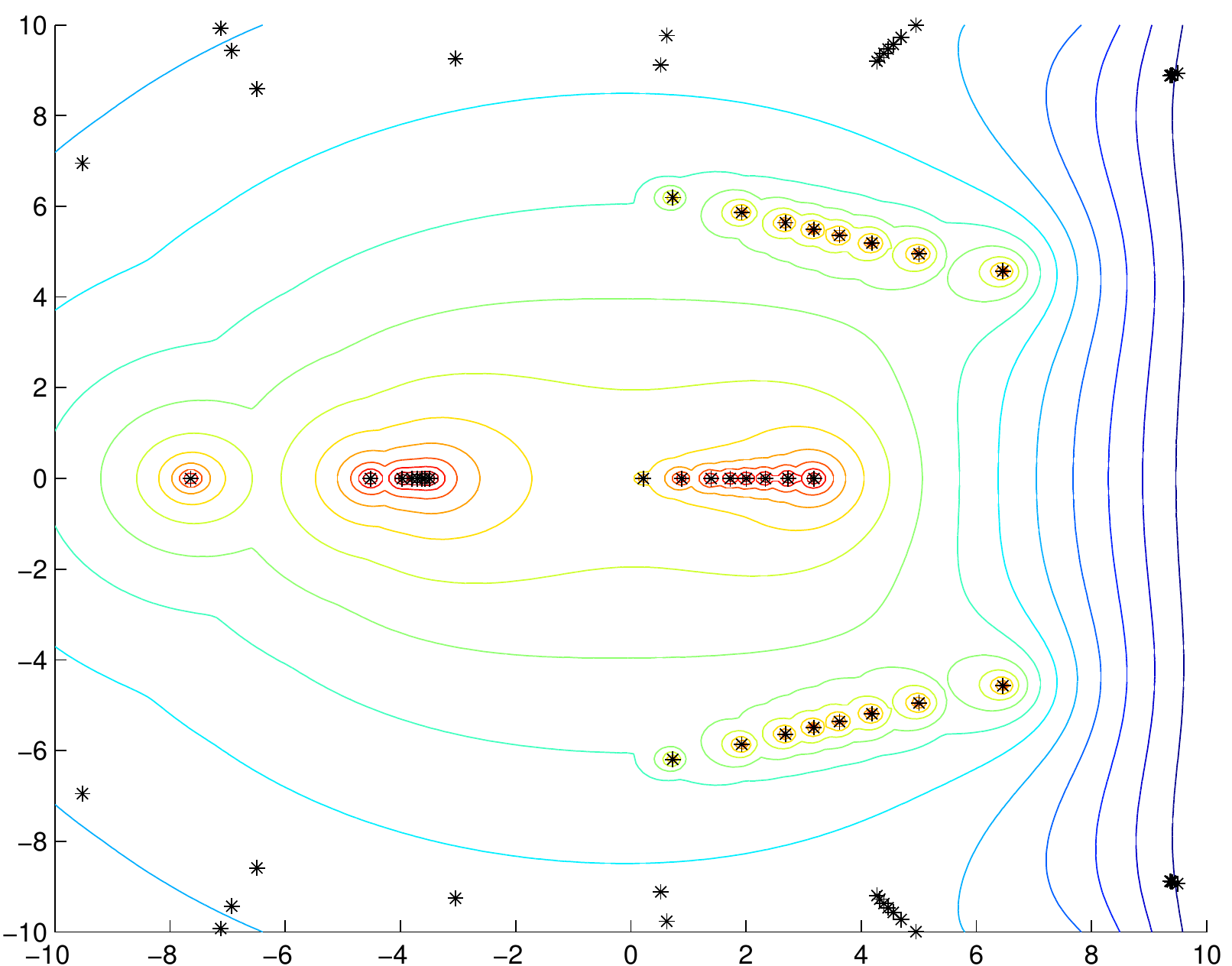}
\end{center}
\caption{Spectrum for a Chebyshev approximation to the Hadeler
         problem (stars), together with the pseudospectrum
         for the Hadeler function.}
\label{fig:hadeler3}
\end{figure}

We first consider the problem of localizing eigenvalues for the
Hadeler example near the real axis.  To do this, we approximate the
Hadeler function
\[
  T(z) = B(\exp(z)-1) + Az - \alpha I
\]
by a polynomial
\[
  P(z) = B q(z) + Az - \alpha I.
\]
where $q(z)$ is the polynomial interpolating $\exp(z)-1$ through
a Chebyshev grid on some interval $[z_{\min}, z_{\max}]$
guaranteed to contain all the eigenvalues, which we obtain
using the Gershgorin bounds from the previous section.

Suppose we write $P(z) = Q(x)$ where 
$z = (1-x)z_{\min}/2 + (1+x)z_{\max}/2$;
that is, $Q$ is a rescaled version of $P$.
If we expand $Q$ in terms of first-kind Chebyshev polynomials
$T_j$ as
\[
  Q(x) = \sum_{j=0}^n A_j T_j(x),
\]
then, assuming $A_n$ is invertible, $\det(A_n^{-1} Q(x)) = \det(C-xI)$, where $C$ is
the {\em colleague matrix} linearization~\cite{Effenberger:2012:CIN}:
\[
  C =
  \frac{1}{2}
  \begin{bmatrix}
  0 & 2I &  \\
  I & 0 & I \\
      & I & 0 & I \\
      &     & \ddots & \ddots & \ddots \\
            &&        & I & 0 & I \\
            &&        &     & I & 0
  \end{bmatrix} -
  \frac{1}{2}
  \begin{bmatrix}
    \\
    \\
    \\
    \\
    \\
    A_n^{-1} A_0 & A_n^{-1} A_1 & \ldots & A_n^{-1} A_{n-1}
  \end{bmatrix}.
\]
Note that if $\lambda$ is an eigenvalue of $T$, then it corresponds
(after an appropriate rescaling of variables) to an eigenvalue of
\[
  \hat{C} =
  \frac{1}{2}
  \begin{bmatrix}
  0 & 2I &  \\
  I & 0 & I \\
      & I & 0 & I \\
      &     & \ddots & \ddots & \ddots \\
            &&        & I & 0 & I \\
            &&        &     & I & 0
  \end{bmatrix} -
  \frac{1}{2}
  \begin{bmatrix}
    \\
    \\
    \\
    \\
    \\
    A_n^{-1} \hat{A_0} & A_n^{-1} A_1 & \ldots & A_n^{-1} A_{n-1}
  \end{bmatrix}.
\]
where $\hat{A_0}-A_0 = \left( \exp(\lambda)-1-q(\lambda) \right) B$.
Because we have expressed our polynomial in a Chebyshev basis,
the colleague linearization is convenient, but other linearizations
are natural for polynomials expressed in other 
bases~\cite{Amiraslani:2009:Linearization}.
One could also write the spectrum of $T$ in terms of a nonlinear
perturbation to one of these other linearizations,
and this would generally lead to different bounds.

By first balancing and then computing an eigendecomposition,
we find $S$ such that
\[
  S^{-1} C S = D_C.
\]
Furthermore, any eigenvalue $\lambda$ for the fully nonlinear problem
is an eigenvalue of
\[
  S^{-1} \hat{C} S = D_C + r(\lambda) S^{-1} E_0 S,
\]
where $r(\lambda) = \exp(\lambda)-1-q(\lambda)$ is the error in the
Chebyshev approximation and $E_0$ is a block matrix with $A_n^{-1}
B/2$ in the $(n,1)$ block and zeros elsewhere.  Therefore, for any
$\epsilon > 0$, the eigenvalues inside the region where
$|r(z)| < \epsilon$ lie in the union of Gershgorin disks of
radius $\epsilon \rho_j$ about the eigenvalues of $C$, where
$\rho_j$ are the absolute row or column sums of $S^{-1} E_0 S$.
The standard theory for convergence of Chebyshev approximations
tells us that for an appropriate range of $\epsilon$ values,
$|r(z)| < \epsilon$ for $z$ in a Bernstein ellipse whose radius 
depends on $\epsilon$; see~\cite[Chapter 8]{Trefethen:2012:ATA}.

If we apply the above procedure with a degree 20 interpolant on the
interval from $z_{\min} = -7.7650$ to $z_{\max} = 3.3149$, we obtain a
polynomial eigenvalue problem whose eigenvalues are shown in
Figure~\ref{fig:hadeler3}.  The polynomial provides good
estimates for the real eigenvalues, and reasonable
estimates for the first clusters of complex eigenvalues near the real
axis.  The other eigenvalues of the polynomial interpolant 
do not approximate any eigenvalues of $T$.

\begin{figure}
\begin{center}
\includegraphics[width=0.8\textwidth]{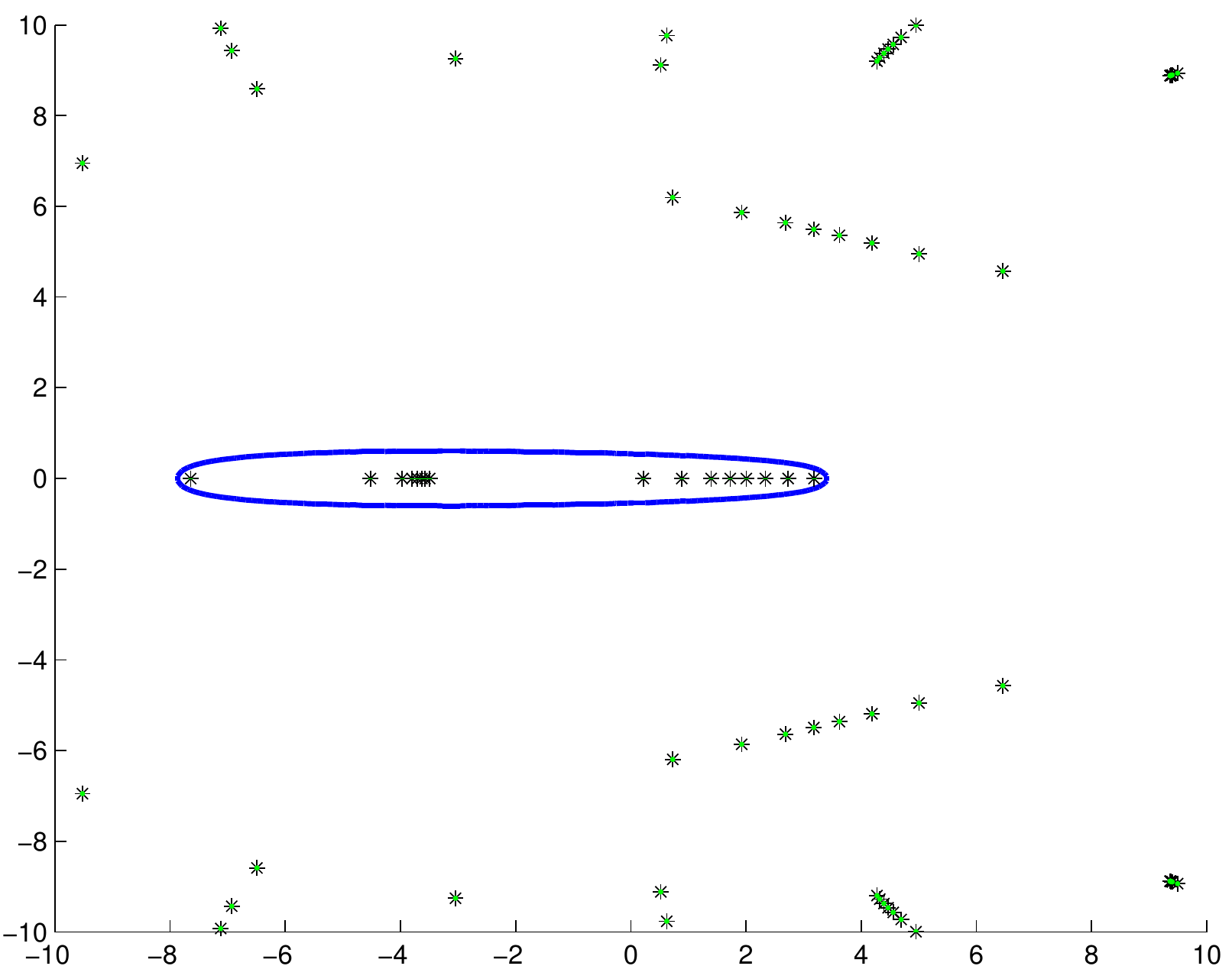}
\end{center}
\caption{Region where the interpolation error $r(z)$ for 
         the Chebyshev approximation to $\exp(z)-1$ is bounded by
         $\epsilon = 10^{-10}$.  The Gershgorin disks in this
         case are all distinct, and all have radii less than $10^{-9}$.}
\label{fig:hadeler4a}
\end{figure}

\begin{figure}
\begin{center}
\includegraphics[width=0.45\textwidth]{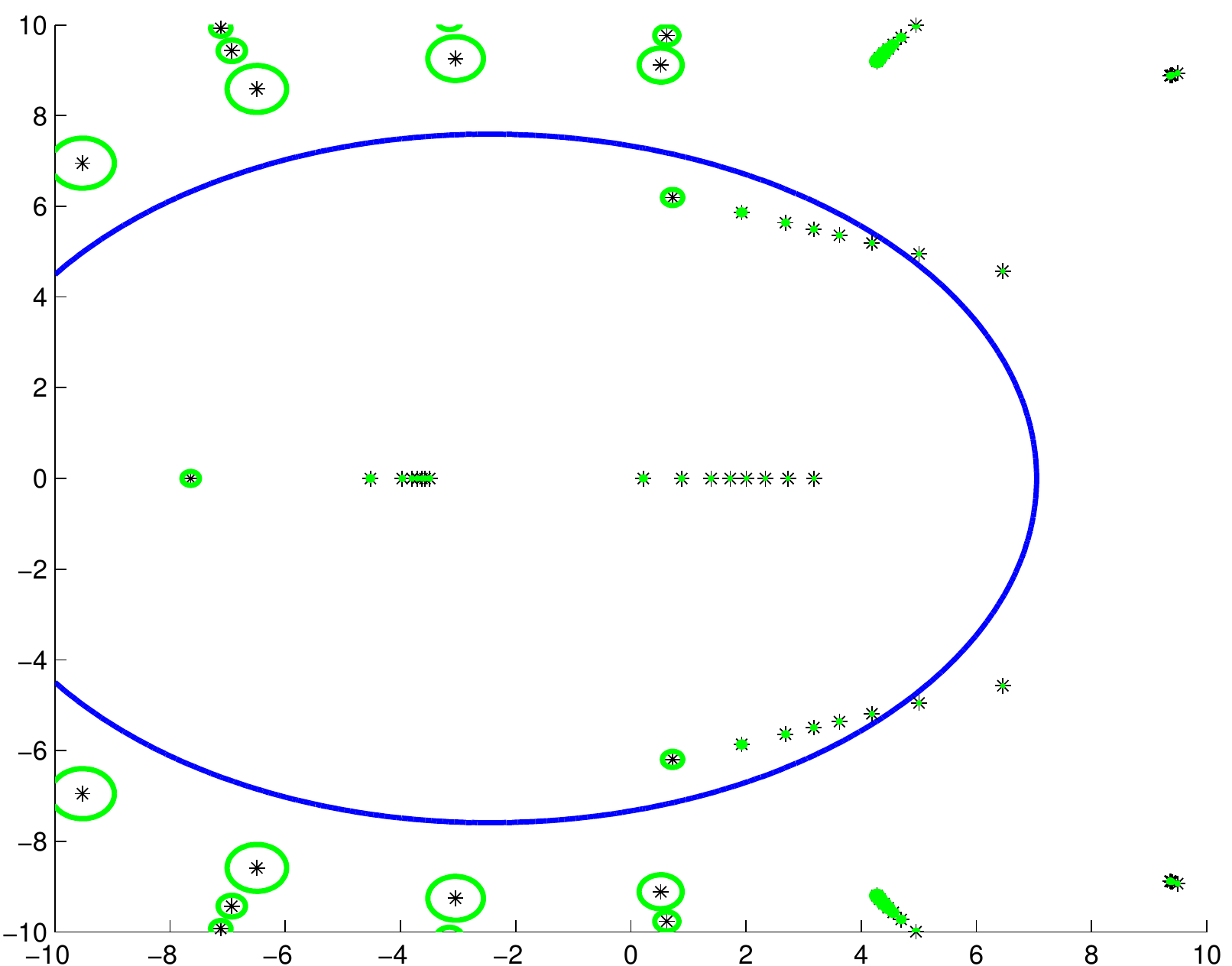}
\includegraphics[width=0.45\textwidth]{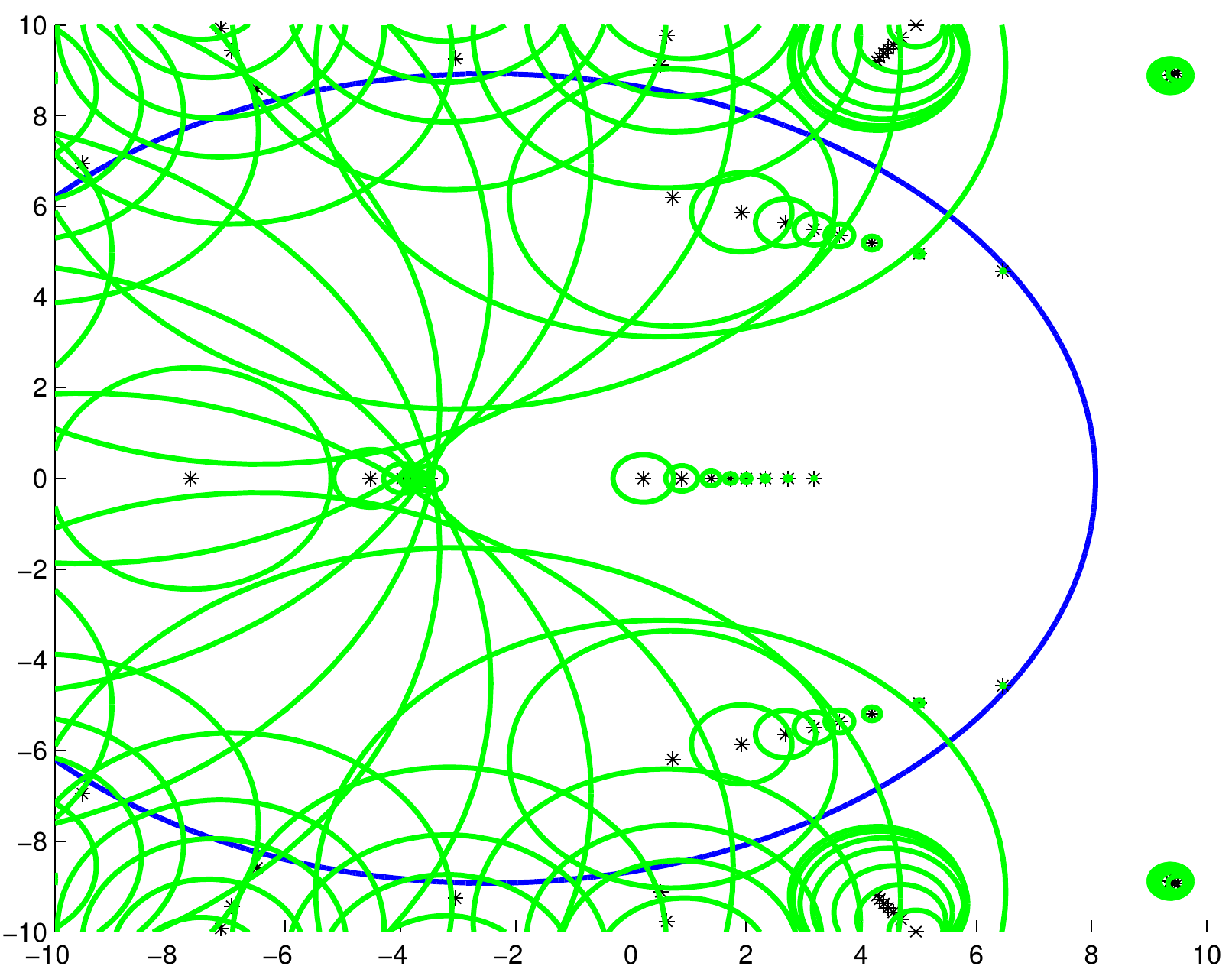}
\end{center}
\caption{Region where the interpolation error $r(z)$ for 
         the Chebyshev approximation to $\exp(z)-1$ is bounded by
         $\epsilon = 0.1$ (left) and $\epsilon = 1.6$ (right).
         The Gershgorin disks of radii $\epsilon \rho_j$ are
         shown in green in each case.}
\label{fig:hadeler4bc}
\end{figure}

For this problem, the largest Gershgorin radius $\epsilon \rho_j$ is
less than $7 \epsilon$.  Figure~\ref{fig:hadeler4a} shows the region
where $|r(z)| < \epsilon = 10^{-10}$; the corresponding Gershgorin
disks in this case are so tight that they are not visible in the
figure.  Thus, we can trust these approximations to the real
eigenvalues to an absolute error of less than $10^{-9}$.  

A more interesting bound involves the eigenvalues farther from the
real axis.  Without the comparison to the previously computed spectrum
of $T$, it would initially be unclear whether the cluster of
eigenvalues with imaginary part near 6 is spurious or not.  If we set
$\epsilon = 0.1$ and $\epsilon = 1.6$, we get the Gershgorin disks
shown in Figure~\ref{fig:hadeler4bc}; these are sufficient to show
that the polynomial eigenvalue clusters closest to the real line also
approximate eigenvalues of $T$, and to bound the approximation error.

\subsubsection{Comparison to a simplified function}
\label{sec:hadeler1}

\begin{figure}
\begin{center}
\includegraphics[width=0.8\textwidth]{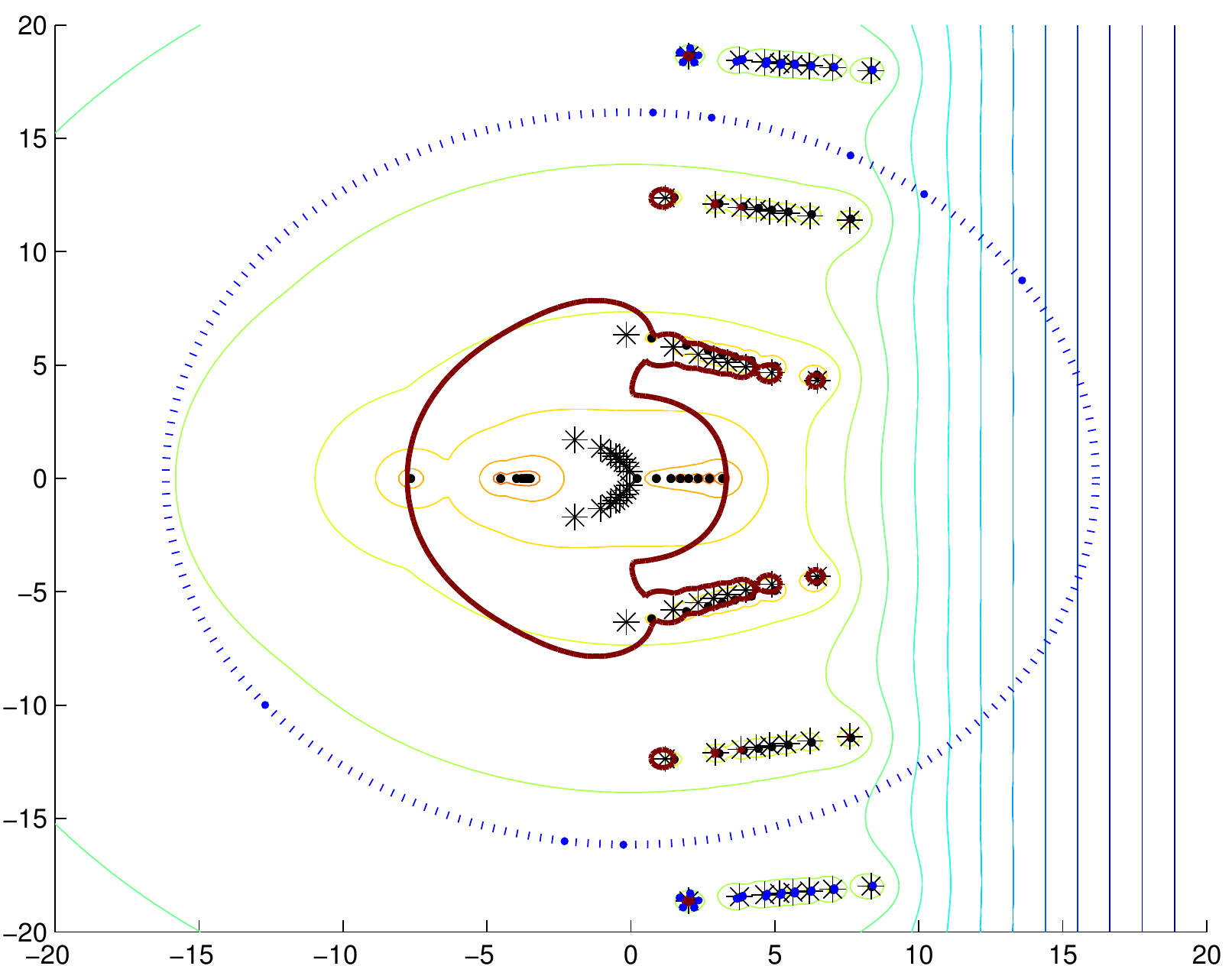}
\end{center}
\caption{Gershgorin region (solid line) containing eigenvalues of the
  Hadeler function $T$ (dots) and the simplified problem $\hat{T}$
  (stars).  Each connected component contains the same number of
  eigenvalues for both problems.  The Gersgorin region is somewhat
  complicated, but it can be shown
  that the components of the Gershgorin regions containing eigenvalues
  $\hat{\lambda} \in \Lambda(\hat{T})$ outside a disk of radius about
  $16.3$ (dashed line) is contained in a union of disks of radius
  $O(|\hat{\lambda}|^{-2})$.  
}
\label{fig:hadeler2}
\end{figure}

The polynomial approximation in the previous section resolves
eigenvalues near the real axis, but tells us nothing about
eigenvalues deeper in the complex plane.  However, for $|z| \gg 1$,
$T(z)$ is dominated by either $B \exp(z)$ or $A z^2$, and the remaining
constant term becomes relatively insignificant.  Therefore, we can
localize eigenvalues far from the origin using a simplified problem
without the constant term.

Let $U$ be a matrix of $A$-orthonormal eigenvectors for
the pencil $(B,A)$, and define
\[
  \tilde{T}(z) = U^T T(z) U
         = D_B \exp(z) + I z^2 + E
\]
where $D_B = \operatorname{diag}(\beta_1, \ldots, \beta_8)$,
$\beta_j > 0$,
and $E = -U^T (\alpha I + B) U$ is a constant matrix.  We
compare $\tilde{T}$ to the simplified function
\begin{align*}
  \hat{T}(z) &= D_B \exp(z) + I z^2
             = \diag \left( f_j(z) \right)_{j=1}^8, \\
  \quad f_j(z) &\equiv \beta_j \exp(z) + z^2 
  = 4 e^z \left( \frac{\beta_j}{4} + \left[ -\frac{z}{2}
    \exp\left( -\frac{z}{2} \right) \right]^2 \right).
\end{align*}
The eigenvalues of $\hat{T}$
lie along the curves 
$z_m(\theta) = (2\theta + (2m-1)\pi) (\cot(\theta) + i)$,
which are the preimage
of $i \bbR$ under the mapping 
$z \mapsto (-z/2) \exp(-z/2)$.
More precisely, the zeros of $f_j$ can be written as
\[
  \hat{\lambda}_{kj}^{\pm} = 
  -2 W_k\left( \pm \frac{i}{2} \sqrt{\beta_j} \right),
  \quad \mbox{ for } k \in \mathbb{Z},
\]
where $W_k$ denotes the $k$th branch of via Lambert's $W$
function~\cite{Corless:1996:lambertW}, the multi-valued solution to the equation
$W(z) \exp W(z) = z$.  

Using Theorem~\ref{th-nl-gerschgorin}, we know that
\[
  \Lambda(\tilde{T}) \subset 
  \bigcup_{j=1}^n G_j^1 \equiv
  \bigcup_{j=1}^n \{ z : |f_j(z)| \leq \rho_j \},
  \quad \rho_j \equiv \sum_{k=1}^n |e_{jk}|.
\]
Furthermore, any connected component of this region contains the same
number of eigenvalues of $\tilde{T}$ and $\hat{T}$.  In
Figure~\ref{fig:hadeler2}, we show a plot of the Gershgorin regions in
the complex plane, with the location of the eigenvalues of $\hat{T}$
marked by asterisks.  These regions are somewhat complicated, but we
can bound some of them in simpler sets.  If $\hat{\lambda}$ is a zero
of $f_j$, then Taylor expansion about $\hat{\lambda}$ yields
\[
  f_j(\hat{\lambda} + w) = bw + R(w),
  \quad b \equiv \hat{\lambda} (2-\hat{\lambda}),
  \quad |R(w)| \leq \left( 1 + (|\hat{\lambda}|^2/2) \exp |w| \right) |w|^2.
\]
If $1 + (|\hat{\lambda}|^2/2) \exp(2\rho_j/|b|) < |b|^2/(4 \rho_j)$, then
$|f_j(\hat{\lambda}+w)| > \rho_j$ for $|w| = 2\rho_j/|b|$.  The
condition always holds for $|\hat{\lambda}| > R \approx 16.3$, and so
the component of $G_j^1$ containing $\hat{\lambda} \in
\Lambda(\hat{T})$ outside this disk must lie in a disk of radius 
$2 \rho_j/|b| = O(|\hat{\lambda}|^{-2})$.
Thus, outside the disk of radius $R$, every eigenvalue of $T$ is
approximated by an eigenvalue of $\hat{T}$ with less than $2\%$
relative error, with better accuracy farther from the origin.

\subsection{Time delay}

\begin{figure}
\begin{center}
\includegraphics[width=0.8\textwidth]{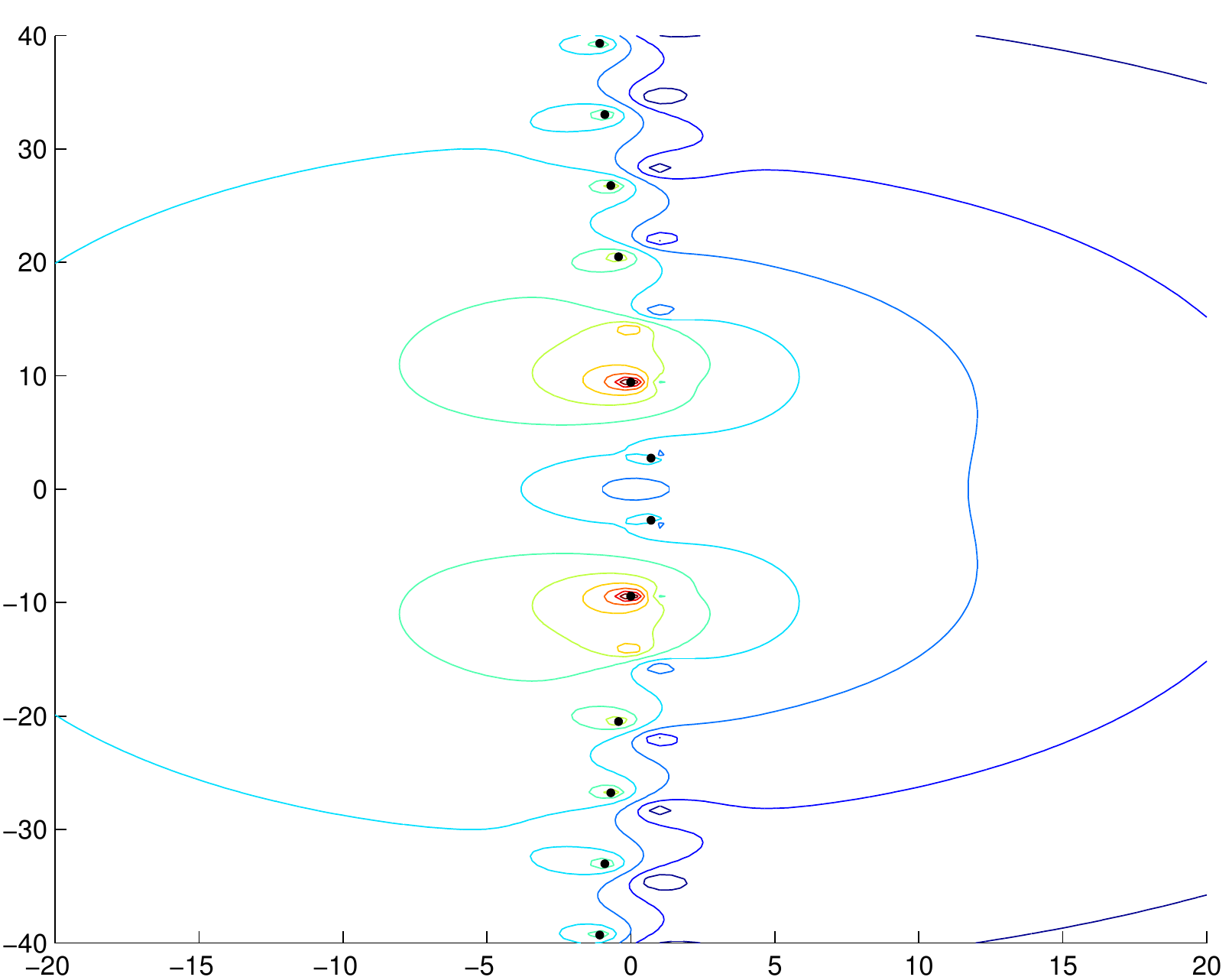}
\end{center}
\caption{Spectrum (in dots) and pseudospectra for the time delay example.
The spectrum closest to the real axis was computed using a degree 40
Chebyshev interpolant of $T$ on the interval $[-12i, 12i]$; farther
out in the complex plane, we get an initial guess from a simplified
problem, then refine using Newton iteration.}
\label{fig:dde1}
\end{figure}

Another example from the NLEVP toolbox is the {\tt time\_delay}
example, which comes from applying transform methods to a delay
differential equation.  The function is
\[
  T(z) = -zI + A_0 + A_1 \exp(-z),
\]
where $A_0$ is a companion matrix and $A_1$ is rank one. 
The spectrum and pseudospectra for this problem over
part of the complex plane are shown in Figure~\ref{fig:dde1}.

\begin{figure}
\begin{center}
\includegraphics[width=0.8\textwidth]{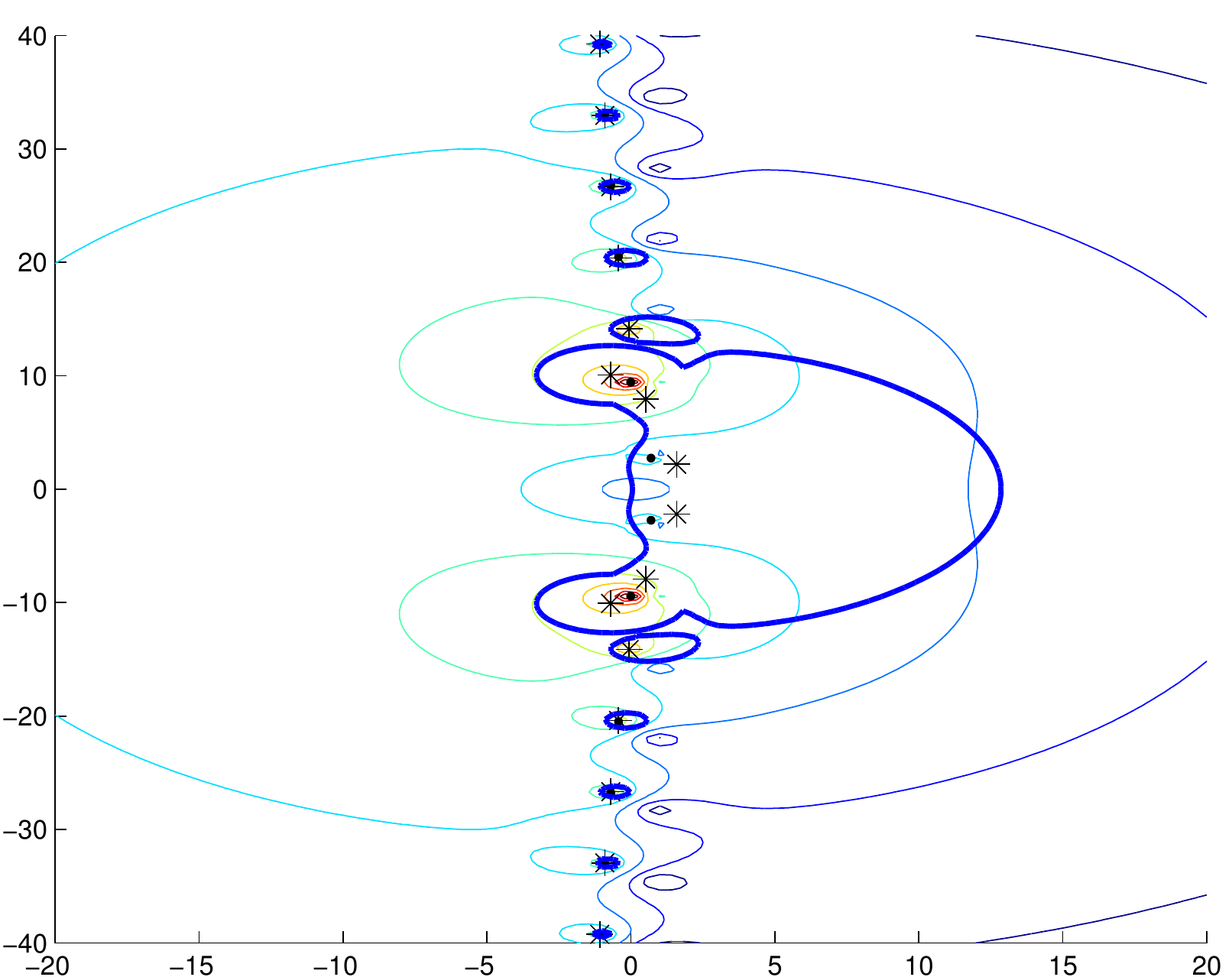}
\end{center}
\caption{Spectrum of the approximation $\hat{T}$ (stars) to the time delay
         problem $T$, Gershgorin regions (thick line), and 
         pseudospectra for $T$.}
\label{fig:dde2}
\end{figure}

As with the Hadeler example, we can get good estimates of the
eigenvalues far from the origin by dropping the constant term in the
problem.  
In order to analyze this case, let us transform $T$ into a
convenient basis.  We choose a basis of eigenvectors $V$ for $A_1$ so
that $V^{-1} A_1 V = D_1 = \diag(\mu_1, 0, 0)$ and so that the
trailing 2-by-2 submatrix of $E = V^{-1} A_0 V$ is diagonal.  The
eigenvalues of $T$ are thus also eigenvalues of
\[
  \tilde{T}(z) = -zI + D_1 \exp(-z) + E;
\]
and, as in the case of the Hadeler example, we can easily compute
the eigenvalues of the related problem
\[
  \hat{T}(z) = -zI + D_1 \exp(-z).
\]
The function $\hat{T}$ has a double eigenvalue at the origin
corresponding to the zero eigenvalues of $A_1$; the remaining
eigenvalues are solutions of the equation
\[
  z \exp(z) = \mu_1 \approx -13.3519,
\]
which can be written as $W_k(\mu_1)$ for $k \in \mathbb{Z}$,
where $W_k$ is again the $k$th branch of Lambert's $W$ function.
Eigenvalues of $\tilde{T}$ must lie in the (column) Gershgorin
region
\[
  \bigcup_{j=1}^3 \{ z : |-z + \mu_j \exp(-z)| \leq \rho_j \},
\]
where $\rho_j$ are the absolute column sums of $E$.
In Figure~\ref{fig:dde1},
we plot this region in the complex plane, with the location of the
eigenvalues of $\hat{T}$ marked by asterisks.  
Theorem~\ref{th-nl-gerschgorin} gives us that each component contains
the same number of eigenvalues of $\tilde{T}$ and $\hat{T}$.  Note in
particular that this means that the central ``blob'' in the
pseudospectrum must contain exactly six eigenvalues of $T$, as indeed
it does -- two eigenvalues closer to the origin, and a pair of
degenerate double eigenvalues at $\pm 3\pi i$; see
\cite{Jarlebring:2012:CNNEP}.

\begin{figure}
\begin{center}
\includegraphics[width=0.8\textwidth]{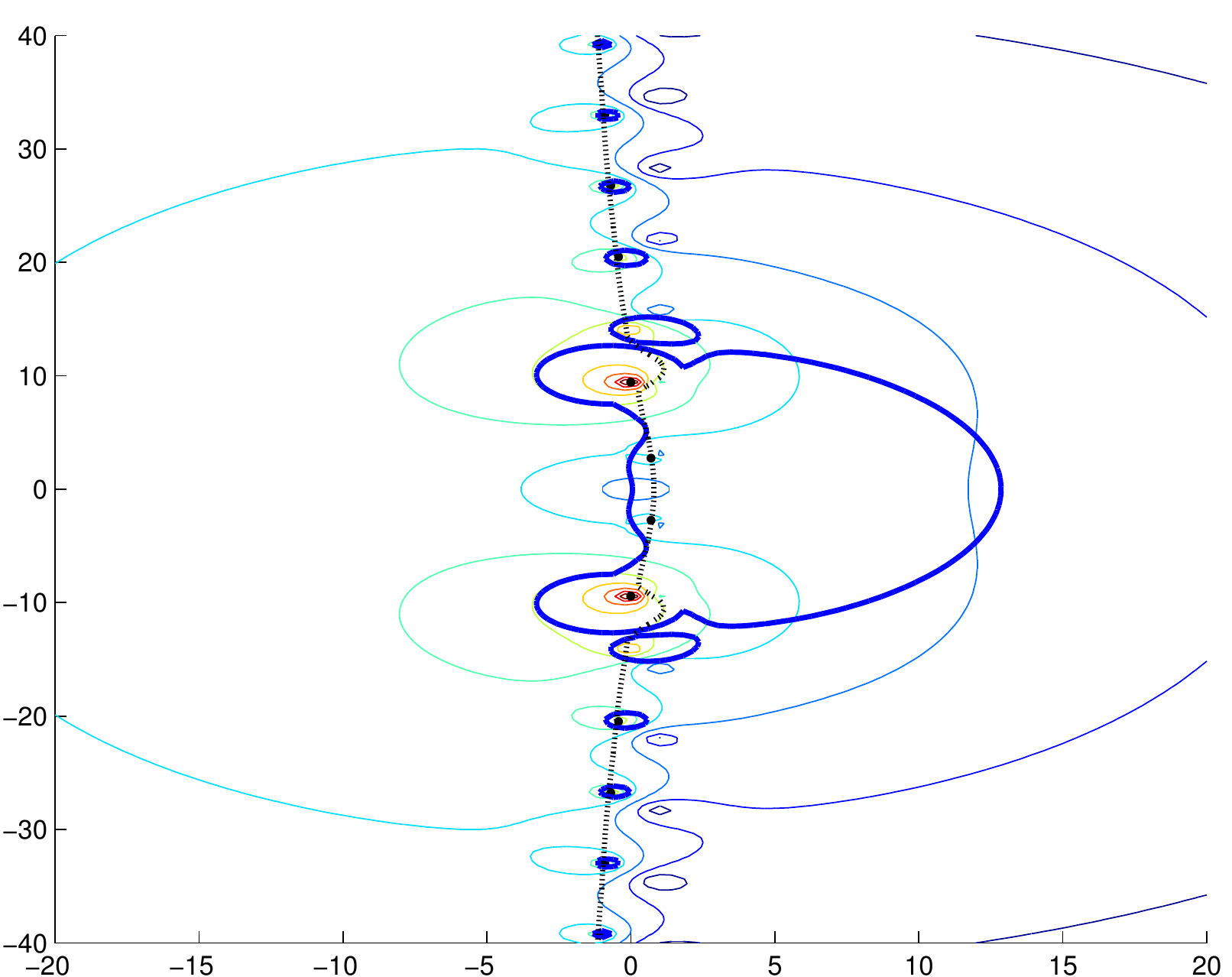}
\end{center}
\caption{Gershgorin region for a simplified problem with the
  exponential term dropped (left of the dashed line).  This is
  superimposed on the Gershgorin regions from Figure~\ref{fig:dde2}.}
\label{fig:dde3}
\end{figure}

The Gershgorin regions shown in Figure~\ref{fig:dde2} obtained by
comparing $T$ to $\hat{T}$ extend far into the right half plane.
We can tighten our inclusion region somewhat by also comparing
$T$ to $A_0-zI$ and letting $A_1 \exp(-z)$ be the error term.
Define
\[
  \breve{T}(z) = V^{-1} T(z) V = D_0-zI + \breve{E} \exp(-z)
\]
where $D_0 = V^{-1} A_0 V$ and $\breve{E} = V^{-1} A_1 V$.  Applying
Theorem~\ref{th-nl-gerschgorin}, any eigenvalue of $T$ must live in
the union of the regions $|d_i-z| \leq \gamma_i |\exp(-z)|$, where
$\gamma_i$ is an absolute row sum of $\breve{E}$.  This region is
bounded from the right by the contour shown in Figure~\ref{fig:dde3}.
Intersecting these bounds with the previous Gershgorin bounds give
very tight control on the spectrum.

\begin{remark} \rm
The determinant of $T(z)$ above is exactly the type of
scalar function studied in~\cite{Bellman:1963:DDE}, and there are similarites between
the analysis done there and in this section, i.e., dropping the
constant term.
\end{remark}
\subsection{Resonances} 

Our final example is a problem associated with resonances of a
Schr\"odinger operator on the positive real
line~\cite{Zworski:1999:RIP,Bindel:2006:ROD}.  We seek the values of
$\lambda$ such that the following two-point boundary value problem has
nontrivial solutions:
\begin{align}\label{pde}
\begin{split}
 \left(-\frac{d^2}{dx^2} + V - \lambda\right)\psi = 0 
 \qquad \text{on } (0,b),\\ 
 \psi(0) = 0
 \qquad\text{and}\qquad 
 \psi'(b) = i\sqrt{\lambda}\psi(b), 
 \end{split}
\end{align}
where $V$ equals $V_0 > 0$ on $(a,b)$ and is zero elsewhere.  In our
computations, we used $(a,b) = (2,3)$ and $V_0 = 5$.  We
formulate~\eqref{pde} as a finite-dimensional nonlinear eigenvalue
problem by shooting from $0$ to $a$ and from $a$ to 
$b$~\cite[Chapter 7]{Ascher:1998:CMO}, described as follows.

Rewriting~\eqref{pde} in first-order form, we have
\begin{equation} \label{res-first-order}
  \frac{du}{dx} =
  \begin{bmatrix}
    0 & 1 \\
    V-\lambda & 0
  \end{bmatrix}
  u, \mbox{ where }
  u(x) \equiv
  \begin{bmatrix}
    \psi(x) \\
    \psi'(x)
  \end{bmatrix}.  
\end{equation}
Then the matrices
\[
  R_{0a}(\lambda) = 
  \exp\left( a 
  \begin{bmatrix} 
    0 & 1 \\
    -\lambda & 0
  \end{bmatrix}
  \right),
  \quad
  R_{ab}(\lambda) = 
  \exp\left( (b-a) 
  \begin{bmatrix} 
    0 & 1 \\
    V_0-\lambda & 0
  \end{bmatrix}
  \right)
\]
respectively map $u(0) \mapsto u(a)$ and $u(a) \mapsto u(b)$.
Thus,~\eqref{pde} is equivalent to the six-dimensional nonlinear
eigenvalue problem
\begin{equation} \label{res-nep6d}
  T(\lambda) u_{\mathrm{all}} \equiv
  \begin{bmatrix}
    R_{0a}(\lambda) & -I & 0 \\
    0 & R_{ab}(\lambda) & -I \\
    \begin{bmatrix} 
      1 & 0 \\ 0 & 0
    \end{bmatrix} &
    0 &
    \begin{bmatrix}
      0 & 0 \\
      -i\sqrt{\lambda} & 1
    \end{bmatrix}
  \end{bmatrix}
  \begin{bmatrix} u(0) \\ u(a) \\ u(b) \end{bmatrix} = 0.
\end{equation}

In the next section, we derive a rational approximation $\hat{T}
\approx T$ whose linearization $K - \lambda M$ corresponds to a
discretization of~\eqref{pde}.  We then use the eigenvalues of
$\hat{T}$ as starting points to compute eigenvalues of $T$, and
establish by Theorem~\ref{th-nep-pseudospec-inclusion} that this
procedure finds all eigenvalues of $T$ in a region of interest.

\subsubsection{Rational approximation of the resonance problem}

\begin{figure}
\begin{center}
\includegraphics{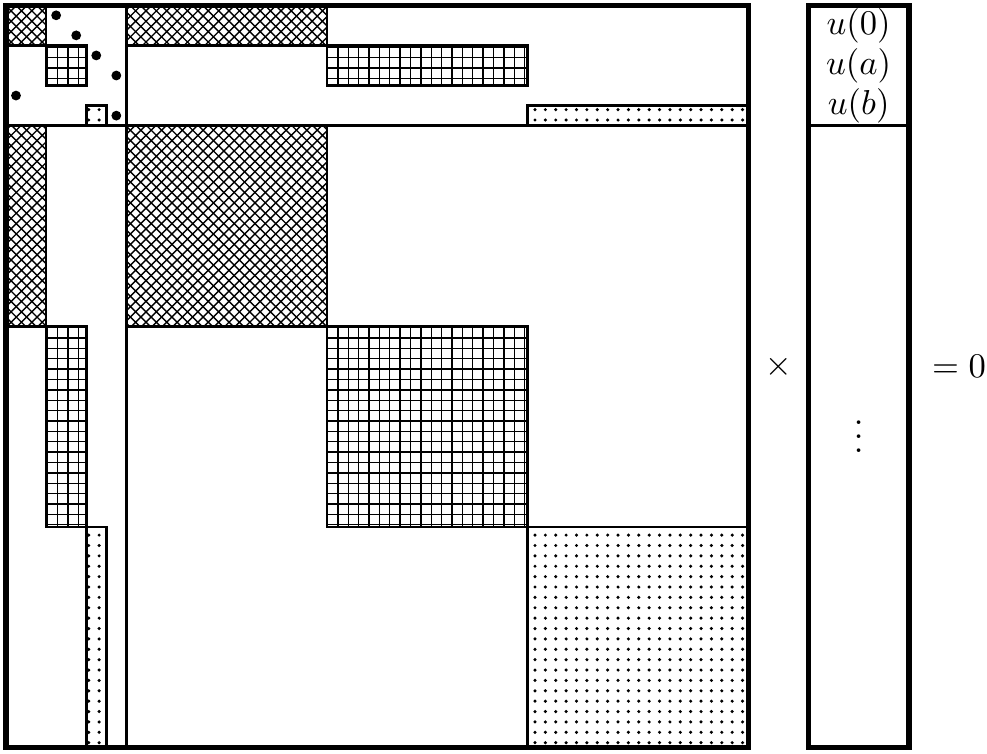}
\end{center}
\caption{Linearized rational approximation to~\eqref{res-nep6d}.  The
  rational eigenvalue problem is a Schur complement in a linear
  eigenvalue problem obtained by eliminating all but the first six
  variables and equations.  The overall matrix is assembled from
  linear matrix-valued functions $A^{0a}(\lambda)$ (cross-hatched), 
  $A^{ab}(\lambda)$ (plaid), and $A^Z(\lambda)$ (dots) that 
  generate rational approximations to $R_{0a}(\lambda)$,
  $R_{ab}(\lambda)$, and $-i \sqrt{\lambda}$, respectively.}
\label{fig:spyrat}
\end{figure}

We construct $K - \lambda M$ by introducing auxiliary variables $y$
whose elimination produces a rational approximation to $T(\lambda)$.
That is, we write $A(\lambda) = K-\lambda M$
so that
\begin{equation} \label{block-A-T-eq}
  \begin{bmatrix}
    A_{11}(\lambda) & A_{12}(\lambda) \\
    A_{21}(\lambda) & A_{22}(\lambda)
  \end{bmatrix}
  \begin{bmatrix}
    u_{\mathrm{all}} \\
    y
  \end{bmatrix} \approx
  \begin{bmatrix}
    T(\lambda) u_{\mathrm{all}}\\
    0
  \end{bmatrix}.
\end{equation}
If we eliminate the auxiliary variables and the equations that
define them, we are left with a rational approximation to 
$T(\lambda)$ given by the leading 
6-by-6
Schur complement in $A$:
\[
  T(\lambda) \approx 
  \hat{T}(\lambda) = 
  A_{11}(\lambda) - A_{12}(\lambda) A_{22}(\lambda)^{-1} A_{21}(\lambda).
\]
More precisely, we will define $A_{11}(\lambda)$ to be the constant
part of $T(\lambda)$, i.e.
\[
  A_{11}(\lambda) =
  \begin{bmatrix}
    0 & -I & 0 \\
    0 &  0 & -I \\
    \begin{bmatrix} 
      1 & 0 \\ 0 & 0
    \end{bmatrix} &
    0 &
    \begin{bmatrix}
      0 & 0 \\
      0 & 1
    \end{bmatrix}
  \end{bmatrix},
\]
then add three submatrices $A^{0a}(\lambda)$,
$A^{ab}(\lambda)$, and $A^Z(\lambda)$ 
(to be defined in a moment)
that generate rational approximations to the nonlinear terms
$R_{0a}(\lambda)$, $R_{ab}(\lambda)$, and $-i\sqrt{\lambda}$
when the Schur complement in $A$ is taken.
The structure of the 
matrix $A$ in terms
of these submatrices is shown schematically in
Figure~\ref{fig:spyrat}.

To define the rational approximation to $R_{0a}(\lambda)$, we start by
writing the exact function $R_{0a}(\lambda) u(0)$ via the equations
\begin{equation} \label{shooting-eq-exact}
  \begin{bmatrix}
    0 & B(a) \\
   -I & B(0) \\
    0 & -\frac{d^2}{dx^2} - \lambda \\
  \end{bmatrix}
  \begin{bmatrix}
    u(0) \\
    \psi
  \end{bmatrix} = 
  \begin{bmatrix}
    R_{0a}(\lambda) u(0) \\
    0 \\
    0
  \end{bmatrix},
\end{equation}
where
\[
  B(x) \psi \equiv \begin{bmatrix} \psi(x) \\ \psi'(x) \end{bmatrix} = u(x).
\]
If we discretize~\eqref{shooting-eq-exact} by replacing
$\psi$ with a vector $\hat{\psi}$ of function values at sample points, 
and correspondingly replace the operators in the second column 
in~\eqref{shooting-eq-exact} with discrete approximations, 
we are left with a matrix equation
\begin{equation} \label{shooting-eq-approx}
  A^{(0a)}(\lambda) 
  \begin{bmatrix}
    u(0) \\
    \hat{\psi}
  \end{bmatrix} \equiv
  \begin{bmatrix}
    0 & \hat{B}(a) \\
   -I & \hat{B}(0) \\
    0 & K_H-\lambda M_H \\
  \end{bmatrix}
  \begin{bmatrix}
    u(0) \\
    \hat{\psi}
  \end{bmatrix} = 
  \begin{bmatrix}
    \hat{R}_{0a}(\lambda) u(0) \\
    0 \\
    0
  \end{bmatrix},
\end{equation}
where $K_H$ and $M_H$ are some fixed matrices of dimension 
$(N-2) \times N$.  For our problem, we set $\hat{\psi}$ to be
function values at a Chebyshev mesh of $N=40$ points on $[0,a]$,
and $(K_H-\lambda M_H)$ represents a pseudospectral collocation
discretization of $-d^2/dx^2-\lambda$~\cite{Trefethen:2000:SMM}.
The matrix $\hat{A}^{ab}(\lambda)$ is defined similarly.

To define $A^Z$, we begin with the best max-norm rational
approximation to $z^{-1/2}$ on an interval $[m,M]$, which was first
discovered in 1877 by Zolotarev~\cite[\S 5.9]{Higham:2008:FM}.  The
approximation is
\[
  z^{-1/2} \approx r(z) = \sum_{j=1}^{N_Z} \frac{\gamma_j}{z-\xi_j}
\]
where the poles $\xi_j$ and the weights $\gamma_j$ are defined
in terms of elliptic integrals; for details, we refer to Method 3
of~\cite{Hale:2008:CAL}.
We approximate $-i\sqrt{\lambda}$ by 
$-i/r(\lambda)$,
which we encode as the leading 1-by-1 Schur
complement in
\[
  A^Z(\lambda) =
  \begin{bmatrix} 
    0 & i \\
    1 & 0 & \gamma_1 & \gamma_2 & \ldots & \gamma_{N_Z} \\
      & 1 & \xi_1-\lambda \\
      & 1 & & \xi_2-\lambda \\
      & \vdots & & & \ddots \\
      & 1 & & & & \xi_{N_Z}-\lambda \\
  \end{bmatrix}.
\]
For our problem, we use the Zolatorev approximation with $N_Z = 20$
poles, chosen for optimality on the interval $[m,M] = [0.1,500]$.

\subsubsection{Analysis of the rational approximation}

Our goal in this section will be to find all eigenvalues in the region
$D$ bounded by the ellipse $\Gamma$ shown in Figure~\ref{fig:err}.
$D$ is clearly contained in $\Omega_{\epsilon}$, where $\epsilon =
10^{-8}$. Moreover, $\Gamma$ was chosen so that $\|T(z)^{-1}\| <
\epsilon^{-1}$ for all $z \in \Gamma$.  This means that the contour
$\Gamma$ does not intersect the $\epsilon$-pseudospectrum of $T$, and
hence any connected component of $\Lambda_{\epsilon}(T)$ in $D$
contains the same number of eigenvalues of $T$ and $\hat{T}$ (by
Theorem~\ref{th-nep-pseudospec-inclusion}). It follows that the same
number of eigenvalues of $T$ and $\hat{T}$ lie in $D$.

Since the norm of the perturbation is small there, we expect that the
eigenvalues of $\hat{T}$ in $D$ are very good approximations to those
of $T$.  We refine each these eigenvalue estimates by Newton iteration
on a bordered system~\cite[Chapter 3]{Govaerts:2000:NMB}.  The
absolute difference between each eigenvalue of $\hat{T}$ and the
corresponding eigenvalue of $T$ is shown in
Table~\ref{tbl:error-bounds}

\begin{figure}
\begin{center}
\includegraphics[width=0.4\textwidth]{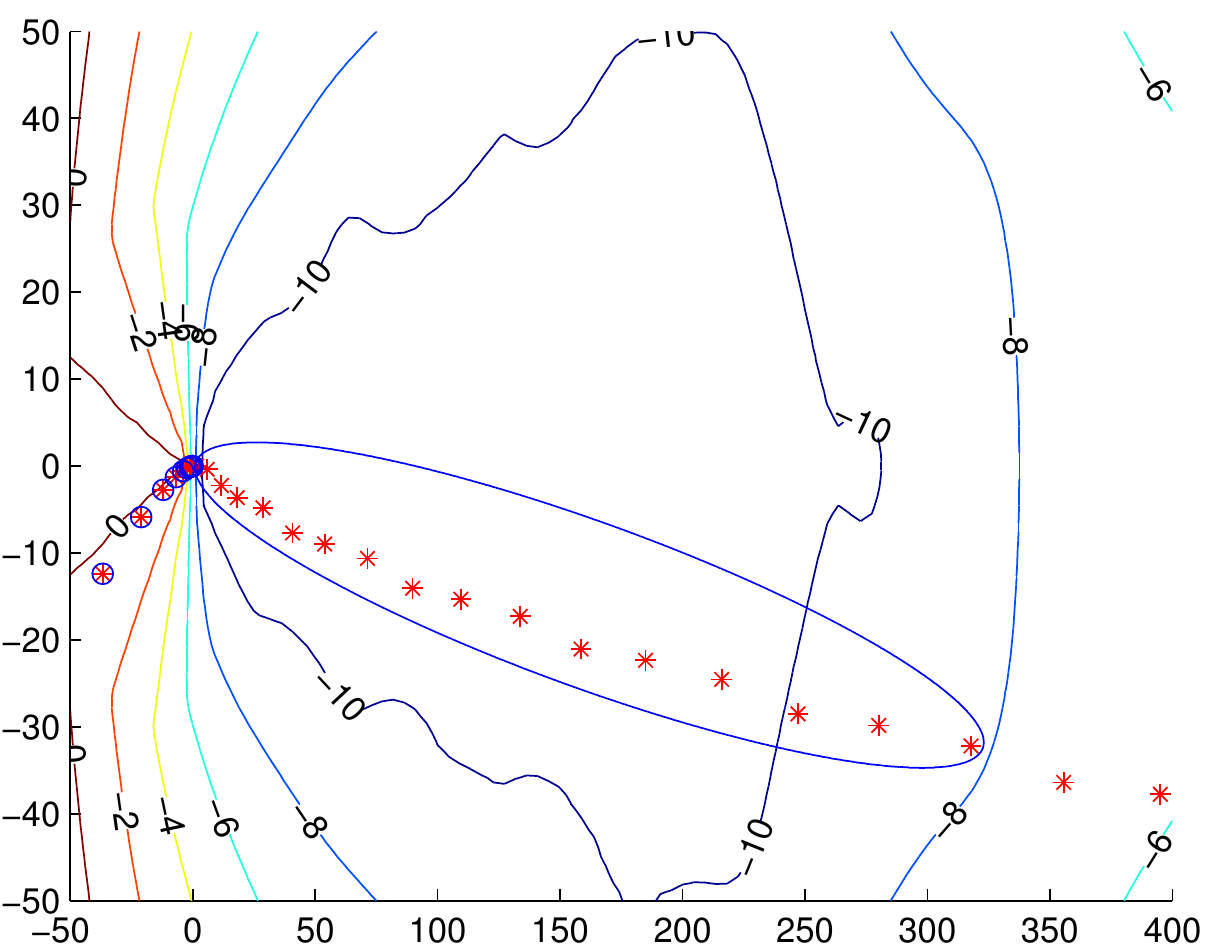}
\includegraphics[width=0.4\textwidth]{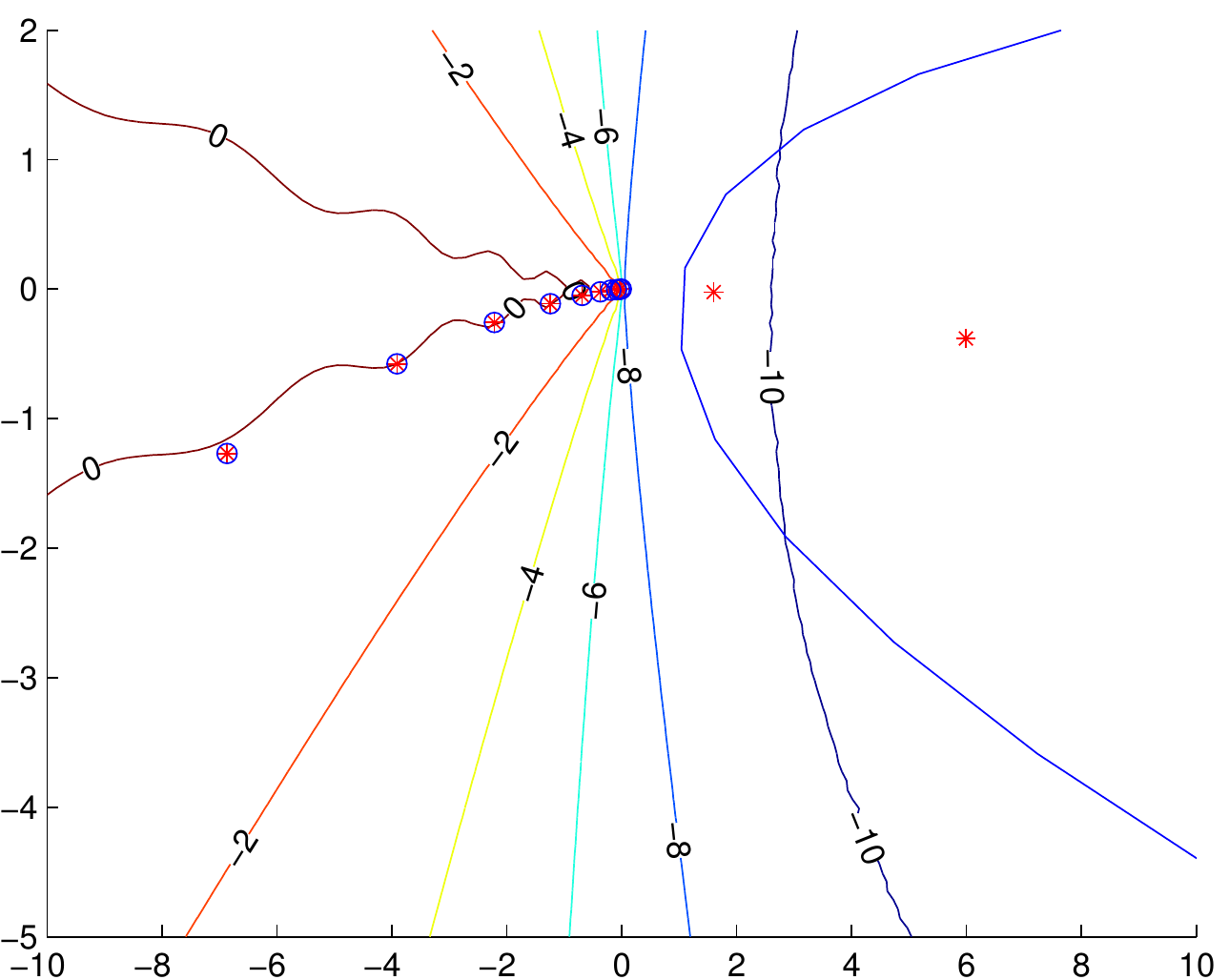}
\end{center}
\caption{Computed eigenvalues for $V_0 = 5$, 
$(a,b) = (2,3)$, 
using
  20 poles in a Zolotarev square root approximation optimal on
  $[0.1,500]$, and 
  Chebyshev meshes of size 40 on $(0,a)$ and $(a,b)$. 
  Circled eigenvalues
  satisfy $\|T(\lambda)\| > 10^{-8}$. Contour plots of
  $\log_{10}(\|T(z)-\hat{T}(z)\|)$  and an ellipse on which
  the smallest singular value of $T(z)$ is greater than $10^{-8}$
  (left). A closer view (right).}
\label{fig:err}
\end{figure}

\begin{table}
\centering
\small
\begin{tabular}{l|l||l|l}
    Eigenvalue & Error & Eigenvalue & Error \\ \hline
    $483.76 -44.65i$ & $1.34 \times 10^{-5}$ &
$439.47 -40.27i$ & $2.39 \times 10^{-6}$ \\
$395.11 -37.76i$ & $4.56 \times 10^{-7}$ &
$355.60 -36.42i$ & $5.67 \times 10^{-8}$ \\
$317.83 -32.22i$ & $7.19 \times 10^{-9}$ &
$280.15 -29.85i$ & $8.45 \times 10^{-10}$ \\
$247.21 -28.52i$ & $1.23 \times 10^{-10}$ &
$215.94 -24.54i$ & $7.10 \times 10^{-11}$ \\
$184.96 -22.34i$ & $9.90 \times 10^{-11}$ &
$158.59 -21.04i$ & $8.45 \times 10^{-11}$ \\
$133.80 -17.31i$ & $4.61 \times 10^{-11}$ &
$109.55 -15.33i$ & $2.20 \times 10^{-11}$ \\
$89.76 -14.06i$  & $5.61 \times 10^{-12}$ &
$71.41 -10.65i$  & $1.60 \times 10^{-11}$ \\
$53.94 -8.99i$   & $2.08 \times 10^{-11}$ &
$40.77 -7.72i$   & $1.97 \times 10^{-11}$ \\
$28.79 -4.80i$   & $5.84 \times 10^{-12}$ &
$18.24 -3.63i$   & $2.04 \times 10^{-12}$ \\
$11.78 -2.23i$   & $3.22 \times 10^{-12}$ &
$5.99 -0.38i$    & $4.75 \times 10^{-13}$ \\
$1.60 -0.02i$    & $1.46 \times 10^{-13}$ \\

\end{tabular}
\caption{Error bounds for computed resonances}
\label{tbl:error-bounds}
\end{table}
\section{Conclusion}
\label{sec-conclusion}

In this paper, we have described several 
localization theorems for the spectrum
of a regular analytic function 
$T : \Omega \rightarrow \bbC^{n \times n}$.  These pseudospectral and
Gershgorin inclusion results generalize well-known perturbation
theory for the standard eigenvalue problem.  We have also shown
through several examples how these results
are practical tools
to localize the 
spectrum, count eigenvalues in parts of
the complex plane, and judge which eigenvalues from an approximating
eigenvalue problem are accurate approximations of true eigenvalues and
which are spurious.

\section*{Acknowledgments}

The authors would like to thank the anonymous referees for their
comments, and in particular for the suggestion to use Lambert's $W$
function in the analysis of the Hadeler and time delay examples.

\bibliographystyle{siam}
\bibliography{refs}
 
\end{document}